\documentclass[11pt,twosided,reqno]{amsart}

\usepackage{geometry}
\usepackage[english]{babel}
\usepackage{mathdots}
\usepackage{graphicx}
\usepackage{amsmath,amsfonts,amssymb,amsthm,epsfig,epstopdf,url,array,mathrsfs}
\usepackage[utf8]{inputenc}
\usepackage[T1]{fontenc}
\usepackage{caption}
\usepackage{wrapfig}
\usepackage{array}
\usepackage{multirow}
\usepackage{pifont}
\usepackage{xcolor}
\definecolor{PLBgras}{rgb}{0.06,0.42,0.60}
\usepackage[colorlinks=true,linkcolor=PLBgras, citecolor=PLBgras, bookmarkstype=toc]{hyperref}
\hypersetup{urlcolor=blue}
\usepackage{enumitem}
\usepackage{stmaryrd}
\usepackage{bbold}
\usepackage[colored]{shadethm} 
\usepackage{tikz}
\usepackage{pgf}
\usepackage{listings}
\usepackage{longtable}
\usepackage{manfnt}
\usepackage{pgfplots}
\usetikzlibrary{patterns}
\pgfplotsset{compat=1.14}
\usepackage[all]{xy}
\usepackage{epsfig}
\usepackage{cancel}
\usepackage{cases}
\usepackage{xspace}
\usepackage{xparse}
\usepackage{cleveref}
\usepackage{enumitem}

\definecolor{BB}{RGB}{162, 22, 22}
\definecolor{NS}{RGB}{12,133,23}
\definecolor{PLB}{rgb}{0.06,0.42,0.60}
\definecolor{PLBfonce}{rgb}{0.06,0.42,0.60}
\definecolor{PLBmoyen}{RGB}{176, 216, 232}
\definecolor{PLBpale}{rgb}{0.94,0.965,0.965}

\theoremstyle{definition}
\newtheorem{deff}{Definition}

\newtheorem*{assump}{Assumptions}

\theoremstyle{plain}
\newtheorem{prop}[deff]{Proposition}
\newtheorem{thrm}[deff]{Theorem}
\newtheorem{coro}[deff]{Corollary}
\newtheorem{lemma}[deff]{Lemma}
\newtheorem{procedure}[deff]{Method}

\theoremstyle{remark}

\newtheorem{rem}[deff]{Remark}

\newcommand{\Z}{\mathbb{Z}}
\newcommand{\N}{\mathbb{N}}

\newcommand{\R}{\mathbb{R}}
\newcommand{\D}{\mathbb{D}}
\newcommand{\Dbar}{\overline{\mathbb{D}}}
\renewcommand{\S}{\mathbb S}

\newcommand{\C}{\mathbb{C}}

\renewcommand{\>}{\geqslant}
\newcommand{\<}{\leqslant}

\newcommand{\interval}[2]{\llbracket #1 : #2 \rrbracket}

\newcommand{\ind}{\text{Ind}}

\newcommand{\symbole}{\gamma}
\newcommand{\dt}{{\Delta t}} 			
\newcommand{\dx}{{\Delta x}}			

\newcommand{\Ind}{\mathrm{Ind}}

\newcommand{\fun}[5]{#1 : #2 \in #3 \mapsto  #4\in #5}

\newcommand{\egdef}{\overset{def}{=}}

\newcommand{\project}{\pi}
\newcommand{\DKLindep}{\Delta}
\newcommand{\DKLindepinvers}{\widetilde{\Delta}}
\newcommand{\DKLindepdivided}{\mathring{\Delta}}

\newcommand{\G}{\mathcal E^s}
\newcommand{\U}{\mathcal U}
\newcommand{\courbe}{\DKLindep(\S)}
\newcommand{\Es}{\mathcal E^s(z)}
\NewDocumentCommand{\DKL}{s}{    
\IfBooleanTF{#1}		            
 {\Delta_{0}}{\Delta_{\mathrm{KL}}}}
 
 \newcommand{\nomCa}{D}
 \newcommand{\nomCb}{D'}

 \newcommand{\mult}{\beta} 
 \newcommand{\nbrmult}{M}
 \newcommand{\borneinfintervalspatial}{x_{\sigma}}
 \newcommand{\gap}{\sigma}

\begin{document}


\title[On the stability of totally upwind schemes for the hyperbolic IBVP]{On the stability of totally upwind schemes for the hyperbolic initial boundary value problem}
\author{Benjamin~Boutin \and Pierre~Le~Barbenchon \and Nicolas~Seguin}
\address{Univ Rennes, CNRS, IRMAR - UMR 6625\\F-35000 Rennes, France.}
\email{benjamin.boutin@univ-rennes1.fr}
\email{pierre.lebarbenchon@univ-rennes1.fr}
\address{IMAG, Inria d'Université Côte d'Azur, Univ. Montpellier, CNRS, Montpellier, France}
\email{nicolas.seguin@inria.fr}
\thanks{This work has been partially supported by ANR project NABUCO, ANR-17-CE40-0025 and by Centre Henri Lebesgue, program ANR-11-LABX-0020-0.}
\date{\today}

\begin{abstract}
    In this paper, we present a numerical strategy to check the strong stability (or GKS-stability) of one-step explicit totally upwind schemes in 1D with numerical boundary conditions. 
    The underlying approximated continuous problem is the one-dimensional advection equation.  
    The strong stability is studied using the Kreiss-Lopatinskii theory. We introduce a new tool, the intrinsic Kreiss-Lopatinskii determinant, which possesses remarkable regularity properties. By applying standard results of complex analysis, we are able to elate the strong stability of numerical schemes to the computation of a winding number, which is robust and cheap. 
    The study is illustrated with the Beam-Warming scheme together with the simplified inverse Lax-Wendroff procedure at the boundary.
\end{abstract}

\maketitle




\section{Introduction}


\subsection{Motivations}

The purpose of this work is to establish an efficient numerical strategy to determine whether a given finite difference method on the half line is stable or not. More precisely, the study is focused on a certain subclass of explicit one-step linear finite difference schemes, specified hereafter. We restrict our attention to the approximation of a rightgoing linear advection equation set on the positive real axis:
\begin{align}\label{eq:advection}
    \begin{cases}
     \partial_t u + a \partial_x u = 0, &  t\>0, x \>0,\\
    u(t,0) = g(t), & t\>0, \\
     u(0,x) = f(x), & x\>0,
    \end{cases}
\end{align}
where $u(t,x)\in\R$. 
The velocity is assumed to be positive $a>0$ so that at the inflow boundary located at the point $x=0$, a physical boundary datum $g$ is  prescribed.
\bigskip

Let us first recall some general ideas and historical context. As a central idea in numerical analysis, the Lax equivalence theorem \cite{Lax56} asserts that a consistent scheme is convergent if and only if it is stable. 
Therefore, all along the paper only consistent numerical schemes are considered, and the discussion concentrates only on their stability issues. While the Cauchy-stability for the space-periodic problem is easily handled with the Fourier symbolic analysis and the so-called Von-Neumann stability analysis, the case with boundaries is 
significantly trickier. Indeed, the presence of (unphysical) numerical boundary conditions forms another kind of instabilities.
The normal mode analysis, directly related to the work by Godunov and Ryabenkii \cite{Godunov63}, is the classic way to comprehend those kinds of instabilities. Deepening this analysis with resolvent estimates and Laplace transform leads to the notion of \emph{GKS-stability} \cite{Gustafsson72} (sometimes called \emph{strong stability}, see Definition~\ref{defstabilite} hereafter). This notion is actually the most robust one concerning the stability of initial boundary value numerical methods, since this stability property is stable by perturbations and makes use of the same norms for the solution and for the data itself. 
These features make possible further extensions to more general cases (e.g. nonlinearities), as it is done for the initial boundary value problem in the case of partial differential equations 
\cite{Benzoni06}. 
In this setting, the Kreiss theorem (see Theorem~\ref{thrmKreiss} later) expresses a necessary and sufficient condition for strong stability by the use of the so-called Uniform Kreiss-Lopatinskii Condition.  
When this condition fails, the corresponding instabilities may be interpreted as numerical wave packets with exponential growth in time and/or bad group velocities (see Trefethen~\cite{Trefethen83, Trefethen84}).
Some sketches of the strong stability theory will be unfolded later on, but we refer the interested reader to the 
monograph \cite{Gustafsson08} by Gustafsson and  \cite{Gustafsson13}  by Gustafsson, Kreiss and Oliger for a more complete overview of the GKS-stability theory.
\bigskip

The GKS-stability theory is not used so often in the numerical analysis literature. The reason is  that the Uniform Kreiss-Lopatinskii Condition requires the search for the vanishing points of the Kreiss-Lopatinskii determinant, which is a complex-valued function  defined on $\{|z|\>1\}$. Except for some particular numerical schemes and boundary conditions, this determinant is not known explicitly. 
Indeed, the complexity of the underlying algebra rapidly increases as the size of stencil increase.
As an example, Thuné develop in \cite{Thune86} a software system for investigating the GKS-stability. Nevertheless, the method requires the numerical approximate computation of the roots of some parameterized characteristic polynomial equation, and may be expensive in terms of CPU time.
In order to tackle the stability properties of the discrete initial boundary value problem, some other strategies are available in the literature. 
Among them, the most natural approach is based on the spectral properties of the operator corresponding to the time-iteration  in the numerical scheme. For a large but finite grid of size $J$, it is represented by a matrix $T_J$ of size $J$. It is a banded Toeplitz or a quasi-Toeplitz matrix depending on the boundary conditions under consideration. Beam and Warming \cite{Beam93} study the asymptotic spectra of such matrices in the limit of large $J$.
Roughly speaking, the stability properties are then related to the uniform boundedness of the powers of the matrix $T_J$, known as the Kreiss matrix Theorem \cite[Chap 18]{Trefethen05}.
Nevertheless, the main difficulty is to also guarantee another uniform boundedness property, with respect to the dimension $J$. 
The uniform boundedness is not easy to characterize by spectral properties. Some specialized tools exist to that aim: resolvent estimates and $\epsilon$-pseudospectrum.
For a wide overview of the Kreiss matrix Theorem and its relationship with resolvent estimates and with the central notion of $\epsilon$-pseudospectrum~\cite{Borovykh00, Spijker02}, we refer the reader to the book by Trefethen and Embree \cite{Trefethen05}. Nonetheless, to our knowledge, the link between GKS-instabilities and the pseudospectrum of the family of quasi-Toeplitz matrices associated to a given scheme is still not completely understood.
In the numerical analysis literature, a first attempt thus consists in considering only grids with a large but fixed size $J$. The postulate is that the asymptotic spectral properties are then already available. This strategy has been used by Dakin, Despres and Jouen \cite{Dakin18} for analyzing some specific boundary conditions that we will again consider with our own method in the present paper.

\bigskip

In the present work, the selected strategy is based on the Uniform Kreiss-Lopatinskii Condition and the search of the vanishing points of the corresponding Kreiss-Lopatinskii determinant, that is a function of the complex parameter $z$ defined for $|z|\> 1$. Instead of using the Kreiss-Lopatinskii determinant, we define the \emph{intrinsic Kreiss-Lopatinskii determinant} that shares the same zeros with the Kreiss-Lopatinskii determinant. 
The main result of the paper (Theorem~\ref{thrm:rationalfraction}) yields an explicit formula for the intrinsic Kreiss-Lopatinskii determinant, showing that it is holomorphic on $\{|z|>1\}$. Moreover, the formula does not require the numerical computation of the roots of the associated characteristic equation. 
Thus, this new theoretical result is particularly useful for numerical applications. Indeed, Corollary~\ref{thrm:nbrzerodet} presents a strategy to find the number of zeros of the intrinsic Kreiss-Lopatinskii determinant on the domain $\{|z|>1\}$ using a numerical computation of winding numbers. Hence, this corollary enables the Method~\ref{proc:numericalprocedure} to tackle the stability of the scheme.
The whole study in this paper is restricted to totally upwind schemes, so the consistency order is limited to 2 (see Iserles \cite{Iserles83}). As typical examples, we therefore focus on the classic first-order upwind and Beam-Warming schemes, while the generality of the study comes from the fact that we can take any extrapolation boundary condition using some points of the domain (the precise form of the considered boundary conditions will be set later at equation \eqref{eqbord}). In the paper, 
the numerical examples deal with the inverse Lax-Wendroff 
boundary condition, and the simplified variants of it, as introduced by Tan, Shu and Vilar in \cite{Tan10, Vilar15} and used by Li, Shu and Zhang in \cite{Li16,Li17,Li22} to solve advection and diffusion equations.
These authors consider a stability analysis based either on the Godunov-Ryabenkii algebraic condition, or by the so-called eigenvalue spectrum visualization method. This last method again requires the use of a finite grid and the computation of the eigenvalues for a large banded matrix.

\bigskip 

The outline of the paper is as follows. In the sequel of this introductive section, we describe the main assumptions and the notion of stability into play. In Section~\ref{sec:detKL}, we set up the main tool for our study that is the Kreiss-Lopatinskii determinant and the intrinsic Kreiss-Lopatinskii determinant, then we state our main results. In Section~\ref{sec:proof}, we prove these results relying on  
linear algebra tools and complex analysis results.
Section~\ref{sec:numerical} gathers several examples and numerical experiments for illustrating the efficiency of the proposed strategy.

\subsection{Notations and assumptions}

Throughout this paper we denote $\S= \{z\in \C, |z|=1\}$ the unit circle, $\D= \{z\in \C, |z|<1\}$ the open unit disk, $\mathcal U= \{z\in \C, |z|>1\}$ the exterior domain and $\overline{\mathcal U} = \{z\in \C, |z|\> 1\}$ its closure. For $n<m$, the notation $\interval{n}{m}$ is for the set $\{k\in \N, n\<k\<m\}$.\\

At the discrete level, we consider explicit one-step finite difference methods of the form
\begin{equation}\label{eq:stdscheme}
    U_{j}^{n+1} = \displaystyle \sum_{k = -r}^p a_k U_{j+k}^n,
\end{equation}
with integers $r,p\> 0$. Here, the unknown of the scheme $U_j^n$ is expected to approximate the quantity~$u(n\dt, j\dx)$. The time step $\dt>0$ and the space step $\dx>0$ are usually choosen with respect to some CFL condition $\lambda = a\dt/\dx \<\lambda_{\textsf{CFL}}$ discussed later on.

The \emph{symbol} associated to the scheme \eqref{eq:stdscheme} is defined, for $\xi \in \R$, by
\begin{equation}\label{eq:symbol}
    \symbole(\xi) =\sum_{k = -r}^p a_k e^{ik\xi}.
\end{equation}

The common set of assumptions used hereafter is the following one. 
\begin{assump}
The scheme~\eqref{eq:stdscheme} is 
\begin{enumerate}[label={(H\arabic{*})}, ref={\rmfamily(H\arabic{*})}]
    \setcounter{enumi}{-1}
    \itemsep0.5em
    \item\label{assumption:nondegenerate}
    \emph{non-degenerate}, in the sense that $a_{-r}\neq 0$,
    \item\label{assumption:totallyupwind}
    \emph{totally upwind}, in the sense that $p=0$,
    \item\label{assumption:cauchystab}
    \emph{Cauchy-stable}, meaning that the symbol $\gamma$ satisfies
         $|\symbole(\xi)|\<1$ for all $\xi\in \R$.
    \item\label{assumption:consistency}
    \emph{consistent} and at least first order, meaning that 
        \[\gamma(0)=\sum_{k = -r}^p a_k = 1\text{ and } -i\gamma'(0)=\sum_{k = -r}^p k a_k = -\lambda.\]
\end{enumerate}
\end{assump}

When dealing with the discrete schemes set over the full line $j\in\mathbb{Z}$, the algebraic characterization of the Cauchy-stability classically follows from the Fourier analysis and makes use of the symbol~$\gamma$. This method is known as the Von Neumann analysis (see \cite{Courant28} and \cite{Crank47}). In the scalar case, it reduces to a geometric property concerning the following closed complex curve.

\begin{deff}\label{def:symbolcurve}
    The \emph{symbol curve} $\Gamma$ is the closed complex parametrized curve
    \[
        \Gamma = \{\theta\in[0,2\pi]\mapsto \gamma(\theta)\}.
    \]
\end{deff}

This definition enables a geometric interpretation of the Cauchy-stability assumption~\ref{assumption:cauchystab} reformulated equivalently as the inclusion $\Gamma\subset\overline{\D}$ (see later Figure~\ref{fig:BWsymbol} for the Beam-Warming scheme). In the same vein, the consistency assumption~\ref{assumption:consistency} admits a geometric form through a first order tangency property of $\Gamma$ to the vertical axis at the parameter point $\theta=0$.
\bigskip

{The stability condition~\ref{assumption:cauchystab} can be easily illustrated graphically in the complex plane. In some sense, our goal is to extend this kind of graphical study when including the numerical boundary conditions.}

\bigskip

For solving the Initial Boundary Value Problem (IBVP)~\eqref{eq:advection} with the discrete scheme~\eqref{eq:stdscheme}, $r$ additional ghost points are needed to take into account the left boundary condition and to fully define the discrete approximation. In the theoretical results of the paper, we assume that the values at these ghost points are obtained from a linear combination of the 
first values of the solution close to the boundary and at the same time step, as follows.
\begin{numcases}{} 
    U_{j}^{n+1} = \sum_{k = -r}^0 a_k U_{k+j}^n, & $j\in \N,\ n\in \N$, \label{eqprincip}\\
    U_j^{n} = \sum_{k = 0}^{m-1} b_{j,k} U_k^n + g_j^n,  & $j\in\interval{-r}{-1},\ n\in \N$, \label{eqbord}\\
    U_j^0 = f_j, & $j \in \N$.\label{eqinit}
\end{numcases}
where $ m, r$ are integers, $f_j$ are approximations of the initial condition $f(x_j)$ and $g_j^n$ are numerical data related to the boundary datum $g$.
With the vector notation $U = (U_{-r}^n \cdots U_{m-1}^n)^T$ and $G = (g_{-r}^n \cdots g_{-1}^n)^T$, the boundary equation \eqref{eqbord} reads also equivalently as $BU=G$ with the following matrix
\begin{equation}\label{eq:Bdefinition}
    B \egdef \begin{pmatrix} 
    1 &  & 0 & -b_{-r,0} & \cdots & -b_{-r,m-1} \\
    & \ddots & & \vdots  & & \vdots \\
    0 & & 1 & -b_{-1,0} & \cdots & -b_{-1,m-1} \end{pmatrix}\in \mathcal M_{r,r+m}(\C).
\end{equation}

This class of boundary conditions encompasses the Dirichlet and Neumann extrapoliation procedures~\cite{Goldberg77}, but also the more general simplified inverse Lax-Wendroff procedure (see \cite{Vilar15}, \cite{Li17}, \cite{Dakin18} and Section~\ref{sec:SILWprocedure}). We will focus on these boundary conditions in our numerical examples. More specific treatments at the boundary  exist, as for example absorbing boundary conditions~\cite{Engquist77} and~\cite{Ehrhardt10}, or transparent boundary conditions~\cite{Arnold03} and~\cite{Coulombel19}, however, in general, they do not enter the present framework.


\subsection{Classic results about strong stability}

The GKS-stability theory (see the seminal paper by Gustafsson, Kreiss and Sundström \cite{Gustafsson72}) handles the discrete IBVP \eqref{eqprincip}-\eqref{eqbord}-\eqref{eqinit} with a zero initial data.
We refer the reader to the work by Wu  \cite{Wu95} and Coulombel \cite{Coulombel11} for more recent development on semigroup estimates. 
They extend a stability result for the discrete IBVP \eqref{eqprincip}-\eqref{eqbord}-\eqref{eqinit}, available for zero initial data, to the case of non-zero initial data. The corresponding notions of stability for the boundary problem makes use of 
the following discrete norms:
\begin{equation*}
    \|U_j\|_{\dt}^2 = \sum_{n=0}^{+\infty}\dt |U_j^n|^2 \text{\quad and \quad}\|U\|_{\dx,\dt}^2 = \sum_{n=0}^{+\infty} \sum_{j = -r}^{+\infty}\dt\dx |U_j^n|^2.
\end{equation*}
The latter norm is associated with the space $\ell^2(\{-r,\dots,-1\}\cup \N)$, denoted shortly $\ell^2$.
We are now in position to define the so-called strong stability, for zero initial data.
\begin{deff}[Strong stability]\label{defstabilite}
    The scheme \eqref{eqprincip}-\eqref{eqbord}-\eqref{eqinit} is strongly stable if, taking $(f_j)=0$, there exist $C>0$ and $\alpha_0$, such that for all $\alpha>\alpha_0$, for all boundary data $(g^n_j)$, for all $\dx>0$, for all $n\in \N$, the approximate solution  $(U_j^n)$ satisfies
    \begin{equation}\label{eq:stability}\sum_{j = -r}^{-1} \|e^{-\alpha n \dt} U_{j}\|_{\dt}^2  + \left (\dfrac{\alpha- \alpha_0}{\alpha\dt +1}\right )\|e^{-\alpha n \dt} U\|_{\dx,\dt}^2 \< C  \sum_{j=-r}^{-1} \|e^{-\alpha n \dt} g_j\|_\dt^2.
    \end{equation}
\end{deff}

We warn the reader that $\|e^{-\alpha n \dt} U_{j}\|^2_{\dt}$ 
is an abuse of notation to describe $\sum_{n=0}^{+\infty}\dt e^{-2\alpha n \dt} |U_j^n|^2$, and similarly for $\|e^{-\alpha n \dt} U\|_{\dx,\dt}^2$.

The following Kreiss theorem provides two necessary and sufficient conditions for the strong stability. 
We provide hereafter a condensed formulation of this theorem, obtained from \cite[Thm 5.1]{Gustafsson72} combined with \cite[Lem 13.1.4]{Gustafsson13} or with \cite[Def 2.23]{Gustafsson08}. It makes use of the notions of \emph{eigenvalue} and \emph{generalized eigenvalue} that will be defined later in \Cref{def:eigenvalue} and \Cref{def:geneigenvalue}.

\begin{thrm}[Kreiss]\label{thrmKreiss}
    The following statements are equivalent:
    \begin{enumerate}[label = (\roman*)]
        \item The scheme \eqref{eqprincip}-\eqref{eqbord}-\eqref{eqinit} is strongly stable in the sense of Definition~\ref{defstabilite}.
        \item The scheme \eqref{eqprincip}-\eqref{eqbord}-\eqref{eqinit} has neither eigenvalue nor generalized eigenvalue.
        \item The Uniform Kreiss-Lopatinskii Condition is satisfied.
    \end{enumerate}
\end{thrm}

The Uniform Kreiss-Lopatinskii Condition corresponds to the absence of zeros for the so-called Kreiss-Lopatinskii determinant (see later~\Cref{def:detKL} and~\cite{Gustafsson13}). These zeros
are identified to eigenvalues or to generalized eigenvalues in the sense of Definitions \ref{def:eigenvalue} and \ref{def:geneigenvalue} and correspond to modal instabilities.
Our numerical analysis of the strong stability of the discrete IBVP will be based on a geometrical study of the Kreiss-Lopatinskii determinant.


\section{Kreiss-Lopatinskii determinants}\label{sec:detKL}

In this section, we introduce the Kreiss-Lopatinskii determinant, define the intrinsic Kreiss-Lopatinskii determinant and construct an algebraic reformulation of it (see Theorem~\ref{thrm:rationalfraction} later). This explicit formula shows that it is holomorphic on $\{|z|>1\}$ and is independent of the roots of the associated characteristic equation. 
At last, by Corollary~\ref{thrm:nbrzerodet}, a numerical procedure based on the Theorem \ref{thrmKreiss} (Kreiss) gives a strategy to tackle the stability of the scheme.

\subsection{Stable subspace $\Es$ and matrix representation}

First, we assume~\ref{assumption:totallyupwind} and study the solutions to the interior equation:
\begin{equation}\label{eqprincip2}
U_{j}^{n+1} = \sum_{k = -r}^0 a_k U_{k+j}^n,\  j\in \N,\ n\in \N.
\end{equation}
To study this equation, the $\mathcal Z$-transform (see \cite[Lesson 40]{Gasquet13}) is applied. This transformation is defined for $(x_n)_{n\in \N}\in\ell^2(\N)$ such that $x_0 = 0$ and $z\in\U$ by $\widetilde{x}(z) = \sum_{n\>0} z^{-n}x_n$. The previous equation then becomes
\begin{equation}\label{eq:Ztransform}
    z\widetilde{U}_j(z) = \sum_{k = -r}^0 a_k\widetilde{U}_{j+k}(z),\ j\in\N,\ z\in\U.
\end{equation}

To solve the linear recurrence equation~\eqref{eq:Ztransform}, let us introduce the following characteristic equation where $z$ plays the role of a parameter and $\kappa$ is the indeterminate:
\begin{equation}\label{eq:eqcharac}
    z\kappa^r = \sum_{k = -r}^0 a_k \kappa^{r+k}.
\end{equation}
\clearpage
This equation is nothing but the discrete dispersion relation of the finite difference scheme~\eqref{eqprincip2}, with frequency parameter $\kappa$ in space and $z$ in time. It is formally obtained 
by looking for solutions to the interior equation \eqref{eqprincip2} having the form $U_j^n = z^n \kappa^j$.
\medskip

\medskip

In the spirit of a classic result by Hersh~\cite{Hersh63}, the following lemma indicates a property of separation for the roots with respect to the unit circle.

\begin{lemma}[Hersh]\label{lem:hersh}
    Assume \ref{assumption:nondegenerate} and \ref{assumption:totallyupwind}. For $z$ in the unbounded connected component of $\C\setminus \Gamma$,  all the roots of the characteristic equation \eqref{eq:eqcharac} are in $\D$.
\end{lemma}

The proof of this result is omitted but may be found in~\cite{Hersh63}.

\begin{rem} Under the Cauchy-stability assumption~\ref{assumption:cauchystab}, the inclusion $\Gamma\subset \overline{\D}$ is known. From there, it follows that the unbounded connected component of $\C\setminus\Gamma$ contains the whole set~$\U$ so that a weaker form of the lemma is available for considering $z \in \U$ only.
If in addition, the considered scheme is also \emph{dissipative}, that is if its symbol $\symbole$ satisfies
\[|\symbole(\xi)|\< 1 - \delta |\xi|^{2s},\quad \xi \in [-\pi,\pi],\]
for some $\delta>0$ and an integer $s\in \N^*$ independent of $\xi$,
then the same separation result is available for $z\in\overline{\U}\setminus\{1\}$. The reason for this property is that one has $\S\cap\Gamma=\{1\}$.
\end{rem}

Lemma~\ref{lem:hersh} (Hersh) is illustrated in Figure~\ref{fig:hersh}.
The first two columns correspond to the Hersh lemma and the third one describes the possible configuration for $z\in\Gamma\cap\S$, typically not meeting the assumptions. This case will be the object of a subsequent discussion.

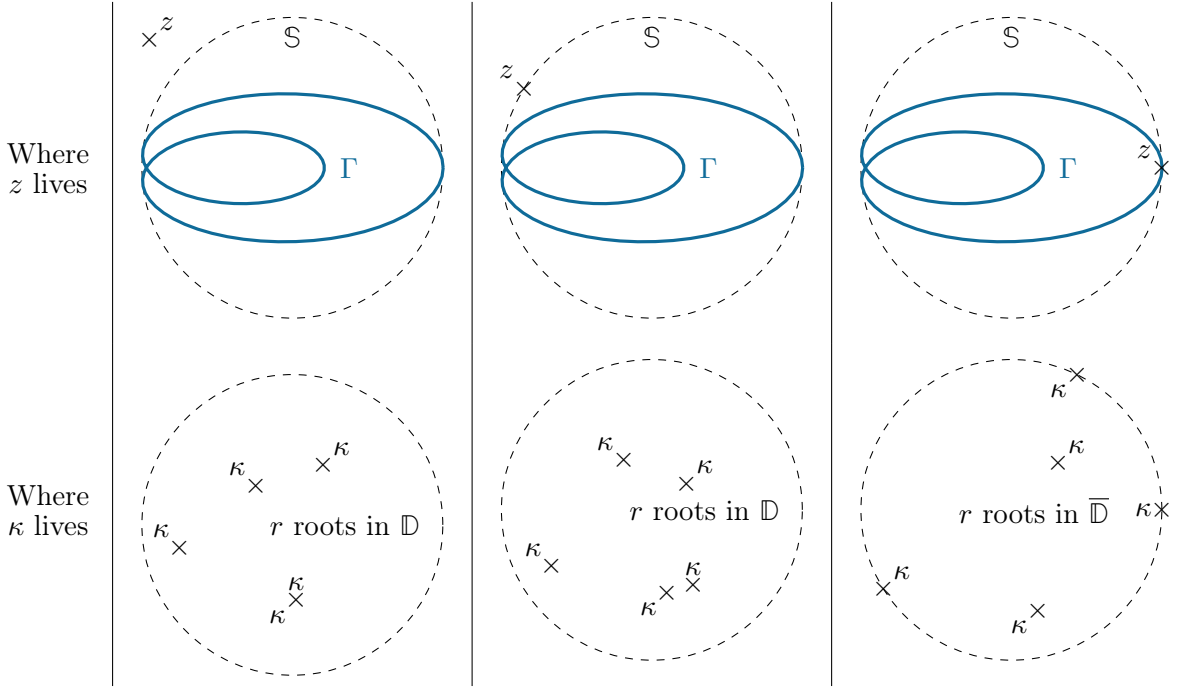
\begin{figure}
    \centering
    \begin{tabular}{c|c|c|c}
        \begin{tikzpicture}
            \draw[white, fill] (-0.1,-2.2) rectangle (0.1,2.2);
            \draw  (0,0.2) node{Where};
            \draw  (0,-0.2) node{$z$ lives};
        \end{tikzpicture}
        &\begin{tikzpicture}
            \draw[white, fill] (-2.2,-2.2) rectangle (2.2,2.2);
            \draw[dashed] (0,0) circle (2);
            \draw[very thick, PLB] plot[domain=0:360,samples=100] ({2.193*(0.72*cos(2*\x) +0.36*cos(\x) - 0.08)-0.188} , {2*(-0.36*sin(2*\x) - 0.18*sin(\x))} );
            \draw  (-1.9,1.7) node{$\times$} node[above right]{$z$};
            \draw  (0.5,0) node[right,PLB]{$\Gamma$};
            \draw  (0,2) node[below]{$\S$};
        \end{tikzpicture}
        &
        \begin{tikzpicture}
            \draw[white, fill] (-2.2,-2.2) rectangle (2.2,2.2);
            \draw[dashed] (0,0) circle (2);
            \draw[very thick, PLB] plot[domain=0:360,samples=100] ({2.193*(0.72*cos(2*\x) +0.36*cos(\x) - 0.08)-0.188} , {2*(-0.36*sin(2*\x) - 0.18*sin(\x))} );
            \draw  (-1.7,1.0535) node{$\times$} node[above left]{$z$};
            \draw  (0.5,0) node[right,PLB]{$\Gamma$};
            \draw  (0,2) node[below]{$\S$};
        \end{tikzpicture}
        &
        \begin{tikzpicture}
            \draw[white, fill] (-2.2,-2.2) rectangle (2.2,2.2);
            \draw[dashed] (0,0) circle (2);
            \draw[very thick, PLB] plot[domain=0:360,samples=100] ({2.193*(0.72*cos(2*\x) +0.36*cos(\x) - 0.08)-0.188} , {2*(-0.36*sin(2*\x) - 0.18*sin(\x))} );
            \draw  (2,0) node{$\times$} node[above left]{$z$};
            \draw  (0.5,0) node[right,PLB]{$\Gamma$};
            \draw  (0,2) node[below]{$\S$};
        \end{tikzpicture}
        
        \\
        \begin{tikzpicture}
            \draw[white, fill] (-0.1,-2.2) rectangle (0.1,2.2);
            \draw  (0,0.2) node{Where};
            \draw  (0,-0.2) node{$\kappa$ lives};
        \end{tikzpicture} &
\begin{tikzpicture}
    \draw[dashed] (0,0) circle (2);
    \draw  (-0.4863821504576905,0.5256329266344337) node{$\times$} node[above left]{$\kappa$};
    \draw  (0.4077563762011518,0.7969276234786205) node{$\times$} node[above right]{$\kappa$};
    \draw  (-1.5,-0.3) node{$\times$} node[above left]{$\kappa$};
    \draw  (0.05,-1) node{$\times$} node[below left]{$\kappa$}
    node[above]{$\kappa$};
    \draw  (0.7,0) node{$r$ roots in $\D$} ;
\end{tikzpicture}
&
\begin{tikzpicture}
    \draw[white, fill] (-2.2,-2.2) rectangle (2.2,2.2);
    \draw[dashed] (0,0) circle (2);
    \draw  (-0.3737011056381775,0.6646464581885967) node{$\times$} node[above left]{$\kappa$};
    \draw  (0.46050826642164544,0.3508791988777962) node{$\times$} node[above right]{$\kappa$};
    \draw  (-1.3332841067875879,-0.7465188753809819) node{$\times$} node[above left]{$\kappa$};
    \draw  (0.2,-1.1) node{$\times$} node[below left]{$\kappa$};
    \draw  (0.55,-1) node{$\times$} node[above]{$\kappa$};
    \draw  (0.7,0) node{$r$ roots in $\D$} ;
\end{tikzpicture}

&
\begin{tikzpicture}
    \draw[white, fill] (-2.2,-2.2) rectangle (2.2,2.2);
    \draw[dashed] (0,0) circle (2);
    \draw  (0.8718,1.8) node{$\times$} node[below left]{$\kappa$};
    \draw  (0.6221277274892463,0.6221277274892463) node{$\times$} node[above right]{$\kappa$};
    \draw  (-1.7,-1.0535) node{$\times$} node[above right]{$\kappa$};

    \draw  (2,0) node{$\times$} node[left]{$\kappa$};
    \draw  (0.35,-1.34360252969993665) node{$\times$} node[below left]{$\kappa$};
    \draw  (0.3,0) node{$r$ roots in $\overline{\D}$} ;
\end{tikzpicture}
\end{tabular}
\caption{Illustration of Lemma \ref{lem:hersh}: case $|z|>1$ (first column), case $|z|=1$ and $z \notin \Gamma$ (second column) and case $z \in \Gamma$ where Lemma \ref{lem:hersh} does not hold (third column).}\label{fig:hersh}
\end{figure}

\begin{rem}\label{rem:hersh}Setting the assumption~\ref{assumption:totallyupwind} aside, meaning with a nonzero number $p$ of right points, the more general form of the Hersh lemma states that for any convenient value of $z$, there are exactly~$r$ roots (with multiplicity) inside the open unit disk, exactly $p$ roots (with multiplicity) outside the unit disk and no root on the unit circle. The result can be proved by using Rouché's theorem.
\end{rem}

For $|z|>1$, we denote $\Es$ the linear subspace of solutions to \eqref{eq:Ztransform} living in $\ell^2$ (the $\ell^2$ space with indices between $-r$ and $+\infty$). By Lemma~\ref{lem:hersh} (Hersh), the space $\Es$ is generated by the following $r$ vectors:

\begin{equation}\label{eq:basis}
     \begin{pmatrix}\kappa_i^{-r} \\ \vdots \\ \kappa_i^{-1} \\ 1 \\ \kappa_i \\\kappa_i^2 \\ \kappa_i^3 \\ \vdots \end{pmatrix}, \begin{pmatrix}-r\kappa_i^{-r} \\ \vdots \\ -\kappa_i^{-1} \\ 0 \\ \kappa_i \\ 2\kappa_i^2 \\ 3\kappa_i^3 \\ \vdots \end{pmatrix}, \dots,\begin{pmatrix}(-r)^{\mult_i-1}\kappa_i^{-r} \\ \vdots \\ (-1)^{\mult_i-1}\kappa_i^{-1} \\ 0 \\ \kappa_i \\ 2^{\mult_i-1}\kappa_i^2 \\ 3^{\mult_i-1}\kappa_i^3 \\ \vdots \end{pmatrix},\quad  i = 1,\dots, \nbrmult\end{equation}
where $\kappa_1, \dots, \kappa_\nbrmult$ of multiplicity $\mult_1, \dots, \mult_\nbrmult$ are the solutions to \eqref{eq:eqcharac}, with $\mult_1 + \cdots + \mult_\nbrmult = r$. 
(We omit the $z$-dependence of $\kappa(z)$ for the sake of readability.)

\medskip

\textbf{Notation.} We denote  $K_{i,j}(z) \in \mathcal M_{j-i+1,r}(\C)$ the matrix where we put in columns the extraction of all the lines between $i$ and $j$ (included) of the previous vectors, where $-r \< i \<j$.

\begin{rem}\label{ex:kappadistinct}
    For $r=2$, if the solutions to \eqref{eq:eqcharac} are $\kappa_1(z) \neq \kappa_2(z)$, then there are exactly two roots with multiplicity~1. The solutions to \eqref{eq:Ztransform} can be written
    $\widetilde{U}_j(z) = \alpha_1\kappa_1(z)^j + \alpha_2 \kappa_2(z)^j,$
    and we have
\[
    K_{-2, 2}(z) = \begin{pmatrix}\kappa_1(z)^{-2}& \kappa_2(z)^{-2} \\\kappa_1(z)^{-1} &\kappa_2(z)^{-1}\\ 1 & 1 \\ \kappa_1(z) & \kappa_2(z)\\\kappa_1(z)^2& \kappa_2(z)^2  \end{pmatrix}.\]
\end{rem}

\begin{rem}\label{ex:kapparacinedouble}
    Still for $r=2$, if the solution to \eqref{eq:eqcharac} now is $\kappa(z)$ with multiplicity 2, then the solutions to \eqref{eq:Ztransform} can be written  $\widetilde{U}_j(z) = (\alpha_1+\alpha_2 j)\kappa(z)^j$, and we have
\[K_{0,3}(z) = \begin{pmatrix} 1 & 0\\ \kappa(z) & \kappa(z)\\ \kappa(z)^2 & 2\kappa(z)^2 \\ \kappa(z)^3 & 3 \kappa(z)^3 \end{pmatrix}.\]
\end{rem}

We raise awareness of the dependence on $z$ and of the continuity issues because the map $z \mapsto K_{i,j}(z)$ is not continuous whereas the set of roots of \eqref{eq:eqcharac} is a continuous mapping with respect to $z$. Indeed, the root curves $(\kappa_j(z))_{j}$ can intersect, when a multiple root occurs. For example, for $r=2$, if there is $(z_n)_{n\in \N}\subset  \U$ with $\kappa_1(z_n) \neq \kappa_2(z_n)$ which converge to $z_{\infty}\in \U$ such that $\kappa_1(z_{\infty}) = \kappa_2(z_{\infty})$ a double root, then we have, for $j = 1$ and $j = 2$,
$$\kappa_j(z_n) \xrightarrow[n\to \infty]{} \kappa_j(z_{\infty})$$ but $$K_{0,3}(z_n) = \begin{pmatrix} 1 & 1 \\ \kappa_1(z_n) & \kappa_2(z_n) \\ \kappa_1^2(z_n) & \kappa_2^2(z_n) \\ \kappa_1^3(z_n) & \kappa_2^3(z_n) \end{pmatrix} \cancel{\xrightarrow[n\to \infty]{}} ~K_{0,3}(z_{\infty}) = \begin{pmatrix} 1 & 0\\ \kappa_1(z_{\infty}) & \kappa_1(z_{\infty}) \\ \kappa_1^2(z_{\infty}) & 2 \kappa_1^2(z_{\infty}) \\ \kappa_1^3(z_{\infty}) & 3 \kappa_1^3(z_{\infty}) \end{pmatrix}.$$

Consequently, the considered basis \eqref{eq:basis} of $\Es$ does not generally define a continuous mapping with respect to $z$. Nevertheless, $\Es$ is a continuous and even holomorphic vector bundle over $\U$ as it is discussed in \cite[Thm 4.3]{Coulombel11}. This author proves in addition that this vector bundle $\Es$ can even be continuously extended over $\overline{\U}$, thus considering $z\in\S$ as well 
(see also \cite{Metivier04} for a similar property for the hyperbolic-parabolic PDE case).
The main point therein is that for some $z_0 \in \S$, there may exist one (or several) root $\kappa_0(z_0)$ of \eqref{eq:eqcharac} on $\S$, because Hersh lemma does not hold anymore. This situation is depicted on the third column of Figure~\ref{fig:hersh} and the different cases that may occur will be explained in Section~\ref{sec:numericalprocedureeig}.

In the case of a totally upwind scheme, it is easy to extend the space $\Es$ because it is the linear space generated by the $r$ roots of \eqref{eq:eqcharac} with polynomial terms for multiplicity. Indeed, $\kappa(z)$ can be defined for all $z\in \overline{\U}$ by continuity of $\kappa(z)$ for $z \in \U$.
The space $\Es$ still is of dimension $r$ and we extend the notation $K_{i,j}(z)$ for $z$ on $\S$. But the difficulty is to prove the continuity of $\Es$ after the extension, it follows from the existence of a K-symmetrizer and is obtained e.g. in \cite[Thm 4.3]{Coulombel11}. As previously observed,  $K_{i,j}(z)$ is generally not continuous with respect to $z$.

We can summarize the discussion in the following theorem.

\begin{thrm}[\cite{Coulombel11}]
    Under assumptions \ref{assumption:nondegenerate}, \ref{assumption:totallyupwind} and \ref{assumption:cauchystab}, the space $\Es$ is a holomorphic vector bundle over $\U$ and can be extended in a unique way as a continuous vector bundle over $\overline{\U}$.
\end{thrm}

Moreover, in the more general case where there are $p$ right points, the extension of $\Es$ is not so easy to define because the $r$ roots that come from the inside of the unit open disk must be selected. Indeed, if there is some $\kappa_0(z_0)$ on the unit circle, one has to know if the root comes from the outside or the inside of the unit disk when $z$ tends to $z_0$ from the outside. Worse, it is possible to have a multiple root on the unit circle with some come from the inside of the unit disk and others from the outside.


\subsection{Intrinsic Kreiss-Lopatinskii determinant}

Now, let us consider the $\mathcal Z$-transformed version of the boundary condition \eqref{eqbord}, that is, for $j$ between $-r$ and $-1$, 
\begin{equation}\label{eqbordZtransform}
    \widetilde{U}_j(z) = \sum_{k = 0}^{m-1} b_{j,k} \widetilde{U}_k(z) + \widetilde{g_j}(z). 
\end{equation}

Injecting the fundamental solutions to $\Es$ into \eqref{eqbordZtransform}, we obtain a system of $r$ equations where the coefficients are the scalar unknowns. They are the coefficients of the solution to~\eqref{eqbordZtransform} written in the basis~\eqref{eq:basis} of $\Es$. 

\begin{rem} For $r=2$ and a given value of $z$ (we skip for convenience the dependence in $z$ hereafter), if $\kappa_1 \neq \kappa_2$ so that the solution to~\eqref{eqbordZtransform} has the form $\alpha_1\kappa_1^j + \alpha_2 \kappa_2^j$, then that solution is constrained by the following two scalar equations:
\[
    \begin{cases}
        \alpha_1\kappa_1^{-2} + \alpha_2 \kappa_2^{-2} = \sum_{k=0}^{m-1} b_{-2,k} (\alpha_1\kappa_1^k + \alpha_2 \kappa_2^k) + \widetilde{g_{-2}},
        \\
        \alpha_1\kappa_1^{-1} + \alpha_2 \kappa_2^{-1} = \sum_{k=0}^{m-1} b_{-1,k} (\alpha_1\kappa_1^k + \alpha_2 \kappa_2^k) + \widetilde{g_{-1}}.
    \end{cases}
\]
The matricial form of that system reads
\[
    \underbrace{\begin{pmatrix} 1 & 0 & -b_{-2,0} & \cdots & -b_{-2,m-1} \\ 0 & 1 & -b_{-1,0} & \cdots & -b_{-1,m-1} \end{pmatrix}}_{B} \begin{pmatrix} \kappa_1^{-2} & \kappa_2^{-2} \\  \kappa_1^{-1} & \kappa_2^{-1} \\ 1 & 1 \\ \kappa_1 & \kappa_2 \\ \kappa_1^2 & \kappa_2^2 \\  \vdots & \vdots  \\ \kappa_1^{m-1} & \kappa_2^{m-1}   \end{pmatrix} \begin{pmatrix} \alpha_1 \\ \alpha_2 \end{pmatrix} = \begin{pmatrix} \widetilde{g_{-2}} \\ \widetilde{g_{-1}}
    \end{pmatrix}.
\]
The injectivity, whence invertibility, of the boundary condition is thus directly related to the property $\det BK_{-2,m-1}(z)\neq 0$, where $BK_{-2,m-1}(z)\in \mathcal M_{2,2}(\C)$.
\end{rem}

\begin{deff}[Kreiss-Lopatinskii determinant]\label{def:detKL}
    The \emph{Kreiss-Lopatinskii determinant} is the complex-valued function defined for $|z|\>1$ by
    \begin{equation*}
        \DKL(z) \egdef \det BK_{-r,m-1}(z).
    \end{equation*}
\end{deff}

Before giving the definition let us motivate the \emph{intrinsic Kreiss-Lopatinskii determinant} $\DKLindep$ by the following informal discussion.
The above Kreiss-Lopatinskii determinant is actually not well defined until we order in some way the roots $(\kappa_j(z))_{j=1,\dots,r}$ of~\eqref{eq:eqcharac}. There are two points to notice. The first one is related to crossing roots, already discussed after Remark~\ref{ex:kapparacinedouble}. The second one is that, outside crossing cases, being given any choice for the ordering of the roots (and thus of the vectors of the basis~\eqref{eq:basis} for the vector bundle), there is in general no chance to obtain a holomorphicity property for the components of the matrix $K_{-r,m-1}(z)$ over $\U$. For example, even the roots of $X^2-z$ are not holomorphic w.r.t $z$ because of the logarithm determination. On the other side, any symmetric functions of the roots $(\kappa_j(z))_{j=1,\dots,r}$ however are holomorphic because they can be obtained directly in terms of the coefficients of the polynomial~\eqref{eq:eqcharac}. Therefore, except for crossing roots, the same holds for the quantity $\DKL(z)$ since the matrix $B$ is constant and the determinant itself is a symmetric function.

It is now known that the space $\Es$ is a holomorphic vector bundle over $\U$, continuous over $\overline{\U}$, and thus we should expect the same for $\DKL$. A very natural way to reach that property and go beyond the last difficulties consists in dividing $\DKL$ by the quantity $\det K_{0,r-1}(z)$. In this manner, the same permutation or combination of the vectors of the basis~\eqref{eq:basis} is involved for both computations.

\begin{deff}[Intrinsic Kreiss-Lopatinskii determinant]
    The \emph{intrinsic Kreiss-Lopatinskii determinant} is the complex-valued function defined for $|z|\>1$ by:
\begin{equation}\label{eq:dklintrinsic}
    \DKLindep(z)=\dfrac{\DKL(z)}{\det K_{0,r-1}(z)}.
\end{equation}
\end{deff}

To conclude with these definitions, let us state a little more about the \emph{Uniform Kreiss-Lopatinskii Condition}. With the above notations and additionally to the invertibility of $ BK_{-r,m-1}(z)$, it corresponds to the existence of a constant $C>0$ such that for any $z\in\overline{\U}$, any $U\in\Es$ solution to~\eqref{eqbordZtransform} satisfies the uniform estimate
\[
    \|\widetilde{U}\| \< C \|\widetilde{g}\|.
\]
From the Parseval identity for the $\mathcal{Z}$-transform, this inequality yields directly the first necessary half-part of the strong stability estimate~\eqref{eq:stability}. We refer the reader to~\cite{Gustafsson13} for a more detailed presentation.


\subsection{Main results} \label{sec:result}

Theorem~\ref{thrm:rationalfraction} is our main theoretical result. It yields an explicit formulation of the intrinsic Kreiss-Lopatinskii determinant and therefore 
describes its properties. Namely, as a function of $z$, this determinant $\DKLindep$ is holomorphic on $\U$, is continuous on $\overline{\U}$ and depends on $\Es$ but not on the choice of a basis (what justifies the \emph{intrinsic} denomination of that quantity).

\begin{thrm}[Explicit formula of the intrinsic Kreiss-Lopatinskii determinant]\label{thrm:rationalfraction}
    Assume \ref{assumption:nondegenerate}, \ref{assumption:totallyupwind}, \ref{assumption:cauchystab} and \ref{assumption:consistency}. The intrinsic Kreiss-Lopatinskii determinant is given, for $z \in \overline{\U}$, by
    \begin{equation}\label{eq:explicitformula}
        \DKLindep(z) = (-1)^{r(m-r)}\det C(z)\left (\dfrac{a_{-r}}{a_0 - z}\right )^{m-r}\end{equation}
    where $\det C(z)$ is a constructible polynomial of $z$ depending only on the coefficients $(a_j)_{j=-r}^0$ and on the components of $B$.
\end{thrm}

By "constructible polynomial", we mean here that we establish a computable algorithm to get a matrix $C(z)$ and then the polynomial $\det C(z)$. This algorithm, based on a gaussian elimination, is fully described in the proof of Lemma~\ref{lem:algebra}.
In the proof of Theorem~\ref{thrm:rationalfraction}, we will explicitly see the holomorphic property of $\DKLindep$. Another property, important for the forthcoming applications, lies in the next Corollary~\ref{thrm:nbrzerodet} and involves the following important geometrical object:
\begin{deff}
    The \emph{Kreiss-Lopatinskii curve} $\DKLindep(\S)$ is the closed complex parameterized curve
    \[ \DKLindep(\S) = \{ \theta\in[0,2\pi]\mapsto \DKLindep(e^{i\theta})\}.\]
\end{deff}

\begin{coro}[Number of zeros of the intrinsic Kreiss-Lopatinskii determinant]\label{thrm:nbrzerodet}
    Assume \ref{assumption:nondegenerate}, \ref{assumption:totallyupwind}, \ref{assumption:cauchystab} and \ref{assumption:consistency}. If $0 \notin \DKLindep(\S)$ then the equation $\DKLindep(z) = 0$ has exactly $r - \Ind_{\DKLindep(\S)}(0)$ zeros in $\U$.
\end{coro}
Here above and in all the paper, $\Ind_{\DKLindep(\S)}(0)$ denotes the winding number of the origin with respect to the closed oriented curve $\DKLindep(\S)$ (see \cite{Lang13} for a definition of the winding number). This previous corollary is the fundamental piece to the following numerical procedure to tackle stability. Indeed, by the definition~\eqref{eq:dklintrinsic}, the function $\DKLindep$ shares the same zeros with the Kreiss-Lopatinskii determinant $\DKL$, which in turn characterizes the stability with Theorem~\ref{thrmKreiss} (Kreiss).


\subsection{Numerical procedure}\label{sec:numericalprocedureeig}

As already seen in the Theorem~\ref{thrmKreiss} (Kreiss), the strong stability can be characterized by the notion of eigenvalue and generalized eigenvalue for the boundary problem. The definition of generalized eigenvalue is not universal, the following one comes from \cite[Def.12.2.2]{Gustafsson13} but one can also find a slightly different one in \cite[Def 2.2]{Gustafsson08}. The difference will be discussed afterwards.

\begin{deff}[Eigenvalue]\label{def:eigenvalue}
    Let $z$ be a complex number. If $|z|\>1$, $\DKLindep(z) =0$ and the solution $(\widetilde{U}_j(z))_j$ to \eqref{eq:Ztransform} and \eqref{eqbordZtransform} is in $\ell^2$ then $z$ is called an \emph{eigenvalue}.
\end{deff}

\begin{deff}[Generalized eigenvalue]\label{def:geneigenvalue}
    Let $z_0$ be a complex number with $|z_0| = 1$. If $\DKLindep(z_0) = 0$ and the solution $(\widetilde{U}_j(z_0))_j$ to \eqref{eq:Ztransform} and \eqref{eqbordZtransform} is not in $\ell^2$  then $z_0$ is called a \emph{generalized eigenvalue}.
\end{deff}

If $|z|>1$ and  $\DKLindep(z)=0$, it is not possible to have $(\widetilde{U}_j(z))\notin \ell^2$, because by Lemma~\ref{lem:hersh} (Hersh), the $r$ roots of \eqref{eq:eqcharac} that are used to construct $(\widetilde{U}_j(z))$ are in the open unit disk. That's why the definition of generalized eigenvalue concerns only complex values on the unit circle.

Therefore, we can split all cases in four types: 
\begin{enumerate}[label = (\roman*)]
    \item\label{item:nonuniteig} $z$ such that $\DKLindep(z) = 0$ and $|z|>1$.
    \item\label{item:uniteig} $z$ such that $\DKLindep(z) = 0$, $|z| = 1$ and $z \notin \Gamma$.
    \item\label{item:geneig} $z$ such that $\DKLindep(z) = 0$, $|z| = 1$, $z\in \Gamma$ and $(\widetilde{U}_j(z))\in \ell^2$.
    \item \label{item:geninf} $z$ such that $\DKLindep(z) = 0$, $|z| = 1$, $z\in \Gamma$ and $(\widetilde{U}_j(z))\notin \ell^2$.
\end{enumerate}

The types \ref{item:nonuniteig}, \ref{item:uniteig} and \ref{item:geneig} describe all the eigenvalues. 
Indeed, for type \ref{item:nonuniteig} and \ref{item:uniteig}, by Lemma~\ref{lem:hersh} (Hersh), we have $(\widetilde{U}_j(z))\in \ell^2$, because every root $\kappa$ of \eqref{eq:eqcharac} is in the open unit disk. Type \ref{item:nonuniteig} corresponds to the first column of Figure~\ref{fig:hersh} and type \ref{item:uniteig} corresponds to the second column.

Moreover the non-existence of eigenvalue of type \ref{item:nonuniteig} is a necessary condition to have stability. It is called the \emph{Godunov-Ryabenkii} condition, introduced in \cite{Godunov63} and described in \cite{Trefethen84}.

If $z$ is of type \ref{item:geneig} or \ref{item:geninf}, there exists a $\kappa_0(z)$ root of \eqref{eq:eqcharac} on the unit circle because $z \in \Gamma$. This is the situation depicted on the third column of Figure~\ref{fig:hersh}. The distinction between \ref{item:geneig} and \ref{item:geninf} is more subtle and comes from the expression of $(\widetilde{U}_j(z))$ in the basis of $\Es$, where the coefficient(s) in front of the vector(s) related to $\kappa_0(z)$ can  be zero or not.

By the way, let us mention that our definition of generalized eigenvalue, from \cite[Def.12.2.2]{Gustafsson13}, corresponds, as we already said, to type~\ref{item:geninf} whereas the definition from \cite[Def 2.2]{Gustafsson08} combines type \ref{item:geneig} and \ref{item:geninf}.

\medskip

Now, Corollary~\ref{thrm:nbrzerodet} can be reformulated as follows.

\begin{coro}
    Assume \ref{assumption:nondegenerate}, \ref{assumption:totallyupwind}, \ref{assumption:cauchystab} and \ref{assumption:consistency}. If $0 \notin \courbe$ then the scheme has $r- \Ind_{\courbe}(0)$ eigenvalues in $\U$ (type \ref{item:nonuniteig}).
\end{coro}

In particular, the low computational cost of the following procedure is very appealing for the study of parameterised IBVP's, see Section 4.
This corollary enables us to establish an efficient and practical method to study the stability of a given IBVP through Theorem~\ref{thrmKreiss} (Kreiss). In particular, the low computational cost of the following procedure is very appealing for the study of parameterised IBVP's, see Section~\ref{sec:numerical}.

\begin{procedure}[Uniform Kreiss-Lopatinskii Condition check]\label{proc:numericalprocedure} 
There are two different cases:
\begin{itemize}[label = $\bullet$]
    \item if $0 \notin \courbe$, there is neither generalized eigenvalue (type \ref{item:geninf}) nor eigenvalue on the unit circle (type \ref{item:uniteig} and \ref{item:geneig}) and there are $r - \ind_{\courbe}(0)$ zeros of $\DKLindep$ in $\U$ by \Cref{thrm:nbrzerodet} (type \ref{item:nonuniteig}). It follows that if the scheme has no eigenvalue in $\U$ then the scheme is stable. Otherwise there exists an eigenvalue and the scheme is unstable.
    \item if $0 \in \courbe$, then there exists $z_0 \in \S$ such that $\DKLindep(z_0) = 0$.
    \begin{itemize}[label = $\to$]
        \item If $z_0 \in \Gamma$, then $z_0$ is a generalized eigenvalue (of type \ref{item:geninf}) or an eigenvalue of type \ref{item:geneig}.
        \item If $z_0 \notin \Gamma$, then there are two possibilities:
        \begin{itemize}[label = \ding{226}]
            \item first, $z_0$ is in the unbounded connected component of $\C\setminus \Gamma$. By Lemma~\ref{lem:hersh} (Hersh), there is no $\kappa$ on the unit circle, so $z_0$ is an eigenvalue on the unit circle (type \ref{item:uniteig}).
            \item second, $z_0$ is in a bounded connected component of $\C\setminus \Gamma$. Contradiction with the Cauchy-stability because $\Gamma\subset \Dbar$ and $z_0 \in \S$.
        \end{itemize}
    \end{itemize}
\end{itemize}
\end{procedure}
This method does not distinguish between types \ref{item:geneig} and  \ref{item:geninf}. In fact, we only study the presence or absence of instabilities, we do not attempt to determine which type of instability mode is met (see Trefethen~\cite{Trefethen84}).

In summary, by \Cref{thrmKreiss} (Kreiss), if $0 \in \courbe$ then the scheme is not stable, and if $0\notin \courbe$,  \Cref{thrm:nbrzerodet} can be used to conclude that the scheme is stable or not, depending on the value of $r - \ind_{\courbe}(0)$. Some illustrations for the Beam-Warming scheme follow in \Cref{sec:numerical} .


\section{Proof of  Theorem~\ref{thrm:rationalfraction} and Corollary \ref{thrm:nbrzerodet}} \label{sec:proof}

In order to use the residue theorem, the holomorphy of the Kreiss-Lopatinskii determinant is needed. To this end, we want a nicer expression of $\det B K_{-r,m-1}(z)$ the  Kreiss-Lopatinskii determinant. 
Clearly the multiplicativity of the determinant does not apply in the expression $\DKL(z)=\det B K_{-r,m-1}(z)$ since $B$ and $K_{-r,m-1}(z)$ are non-square matrices. 
A first step consists of reducing the problem to a linear algebra formulation with square matrices to use the multiplicativity of the determinant.
All along the current section, the assumptions~\ref{assumption:nondegenerate},~\ref{assumption:totallyupwind} and~\ref{assumption:cauchystab}, required to define the matrices $K_{i,j}$, the vector bundle $\G$ as well as its extension over $\overline{\U}$, are supposed to be fulfilled.


\subsection{Reduction to a square formulation}

Let us fix $z \in \overline{\U}$. We recall that $\G(z)$ denotes the space of solutions $(\widetilde{U}_j(z))_{j \> -r}$ to
\[z \widetilde{U}_j(z) = \sum_{k=-r}^0 a_k \widetilde{U}_{j+k}(z),\]
for all $j \>0$ and with $a_{-r} \neq 0$.

\begin{deff}\label{def:equivalencerel}
    Let $E$ be a linear subspace of $\ell^2(\N)$. Two matrices $B,D \in \mathcal M_{r,N}(\C)$ (with $N\in\N\setminus\{0\}$ be any nonzero integer) are said to be equivalent, which we denote $B \sim_E  D$, if and only if for all $U \in E$, one has $B \project(U) = D \project(U)$, where $\project$ is the canonical projection from $\ell^2$ onto $\C^N$, keeping the $N$ first components of $U$.
\end{deff}

To act conveniently with elementary Gaussian operations, we use some specific notations in the following discussions. We denote $M[i : j,k:\ell]$ the matrix obtained by the extraction of the lines between $i$ and $j$ and the columns between $k$ and $\ell$ of the matrix $M$ (all indices are included). Similarly, we denote more shortly $M[k:\ell]$ for the entire columns between column $k$ and column $\ell$ and $M[k]$ for the column~$k$.

\begin{lemma}\label{lem:algebra}
    Let $N\>r$ be an integer. Let $B \in \mathcal M_{r,N}(\C)$ be a constant complex matrix such that $B[1:r,1:r] \in \mathrm{GL}_r(\C)$. Assume moreover that $|a_0|<1$.\\
    For any $z\in\overline{\U}$, consider the associated linear subspace $\G(z)$. There exists a unique square matrix $C(z) \in \mathcal M_r(\C)$ such that 
        \begin{center}
            \begin{tikzpicture}
                \node (A) at (0,0) 
                    {$B \sim_{\G(z)} \begin{pmatrix}
                    \begin{array}{ccc|c} 0& \cdots& 0& \\ \vdots & &\vdots & \hspace{3mm} C(z)\\ 0& \cdots &0 & \end{array}\end{pmatrix}$};
                    \begin{scope}[xshift=1em]
                \draw[<->,>=latex] (-1.25,-1)-- (0.6,-1) node[midway,below]{$N-r$};
                \draw[<->,>=latex] (0.7,-1)-- (2.25,-1) node[midway,below]{$r$};
                \draw[<->,>=latex] (2.5,-0.75)-- (2.5,0.75) node[midway,right]{$r$};
                    \end{scope}
            \end{tikzpicture}
        \end{center}
    Moreover, the components of $C(z)$ are polynomial functions of $z$ and satisfy $\deg \det C(z) = N-r$.
\end{lemma}

\begin{rem}\label{rem:BK=CK}
    Let us highlight our use of this lemma.
    Let $\ell\>-r$ and $z\in\overline{\U}$ be fixed. From the basis~\eqref{eq:basis}, the columns of the matrix $K_{\ell,\ell + N-1}(z)$ take the form $\project(U)$ for some $U\in\G(z)$ and $\project$ the canonical projection from $\ell^2(\N)$ onto $\C^N$. Therefore, for any convenient matrices $B$ and $D$ with $B \sim_{\G(z)} D$, one has then
    $BK_{\ell,\ell + N-1}(z) = DK_{\ell,\ell + N-1}(z)$.\\
    Now for the boundary matrix $B$ defined in \eqref{eq:Bdefinition} and the matrix $D(z)=(0\ | \ C(z))$ obtained by the lemma, the following computation by block is possible
    \begin{align*}
        \det(B K_{-r,m-1}(z)) & = \det (0 K_{-r,m-r-1}(z) + C(z)K_{m-r,m-1}(z)) \\
         & = \det (C(z)K_{m-r,m-1}(z)) \\
         & = \det C(z) \det K_{m-r,m-1}(z).
    \end{align*}
    In other words, the product $BK_{-r,m-1}(z)$ is written as the product of two square matrices, so that the multiplicativity of the determinant can be applied.
\end{rem}

\begin{proof}[Proof of \Cref{lem:algebra}]
    \ \\
    \textbf{Proof of existence:} we proceed by induction for $j$ going from $0$ to $N-r$. At each step, we construct a matrix $B^{(j)}$ which satisfies the following induction hypotheses:
    \begin{enumerate}[label = (\alph*)]
        \item \label{zerohyp} $B \sim_{\G(z)} B^{(j)}$.
        \item \label{firsthyp} the $j$ first columns of $B^{(j)}$ are zero.
        \item \label{secondhyp} every component of $B^{(j)}$ is polynomial of $z$.
        \item \label{threehyp} every component of $B^{(j)}[r+1+j:N]$ are independent of $z$.
        \item \label{fourhyp} the degree of $\det B^{(j)}[j+1:j+r]$ is $j$.
    \end{enumerate}

    \emph{Initialization:} we define $B^{(0)}\egdef B$ which satisfies the five induction hypotheses. The induction hypotheses from \ref{zerohyp} to \ref{threehyp} are trivially satisfied. The induction hypothesis \ref{fourhyp} is satisfied because, $\det B[1:r,1:r] \in \C^*$ which is a non zero constant polynomial.

    \emph{Induction:} we suppose true the induction hypotheses for some $j\in \interval{0}{N-r-1}$ and we want to prove it for $j+1$.

    Let us define $B^{(j+1)} \egdef B^{(j)} - \widetilde{B^{(j)}}$ where
		\begin{center}
			\begin{tikzpicture}
				\node (A) at (0,0) {$\widetilde{B^{(j)}} \egdef \begin{pmatrix} B^{(j)}_{1,j+1} \\ \vdots \\ B^{(j)}_{r,j+1}\end{pmatrix} \begin{pmatrix} 0 ~ \cdots ~ 0 & 1 & \frac{a_{-r+1}}{a_{-r}} &   \cdots & \frac{a_0-z}{a_{-r}}& 0 ~ \cdots ~ \cdots ~ 0\end{pmatrix}$.};
                \draw[<->,>=latex] (-2,-0.4)-- (-0.9,-0.4) node[midway,below]{$j$};
                \draw[<->,>=latex] (-0.5,-0.4)-- (2.8,-0.4) node[midway,below]{$r+1$};
                \draw[<->,>=latex] (3.3,-0.4)-- (5.1,-0.4) node[midway,below]{$N-(r+1)-j$};
			\end{tikzpicture}
		\end{center}
    By construction of $\widetilde{B^{(j)}}$, we have $\widetilde{B^{(j)}}U = 0$ for all $U \in \G(z)$, because the product of the previous row matrix and every vector $U \in \Es$ is equal to zero. Then, we have $B^{(j+1)} \sim_{\G(z)} B^{(j)}$ and by \ref{zerohyp}$\phantom{}_j$, we have \ref{zerohyp}$\phantom{}_{j+1}$.
    Moreover, by \ref{firsthyp}$\phantom{}_j$, the first $j$ columns of $B^{(j+1)}$ are zero because those columns in the construction of $B^{(j+1)}$ are unchanged and by construction of $B^{(j+1)}$, the ($j+1$)-th column is vanished. Then we have \ref{firsthyp}$\phantom{}_{j+1}$.
    By construction, components of $B^{(j)}$ are added and multiplied by $z$ or by real coefficients, then we have \ref{secondhyp}$\phantom{}_{j+1}$.
    By \ref{threehyp}$\phantom{}_j$, the last $N-(r+1)-j$ columns of $B^{(j+1)}$ are independent of $z$ because we do not take into account those columns in the construction of $B^{(j+1)}$, then we have \ref{threehyp}$\phantom{}_{j+1}$.
    
    Finally, we have to find the degree of $\det B^{(j+1)}[j+2:j+1+r]$. We use the multilinearity and the alternating property of the determinant. We work on block matrices and find
    \begin{equation*}
        \det B^{(j+1)}[j+2:j+1+r] =  \det \left ( B^{(j+1)}[j+2:j+r] ~\Big |~ \hspace{-1mm}B^{(j)}[j+1+r] - \frac{a_0 -z}{a_{-r}} B^{(j)}[j+1]\right ).\end{equation*}
    
    Since the matrix $B^{(j)}[j+1+r]$ is independent of $z$ by hypothesis \ref{threehyp}$\phantom{}_j$, 
    the degree of the polynomial $\frac{a_0 -z}{a_{-r}}\det \begin{pmatrix}B^{(j+1)}[j+2:j+r]\mid B^{(j)}[j+1]\end{pmatrix}$ 
    is greater than the degree of $\det \begin{pmatrix}B^{(j+1)}[j+2:j+r]\mid B^{(j)}[j+1+r]\end{pmatrix}$, then it is sufficient to find the degree of  $\frac{a_0 -z}{a_{-r}}\det \begin{pmatrix} B^{(j+1)}[j+2:j+r]\mid B^{(j)}[j+1]\end{pmatrix}$.
    Moreover, the $k$-th column of $B^{(j+1)}[j+2:j+r]$ for $k \in \interval{1}{r-1}$ is $B^{(j)}[j+1+k] - \frac{a_{-r+k}}{a_{-r}} B^{(j)}[j+1]$. 
    
    Then, by alternating property of the determinant, we have
    \begin{align*} 
        & - \frac{a_0 -z}{a_{-r}} \det \begin{pmatrix}B^{(j+1)}[j+2:j+r]\mid B^{(j)}[j+1]\end{pmatrix}\\
        = & - \frac{a_0 -z}{a_{-r}} \det \begin{pmatrix}B^{(j)}[j+2:j+r]\mid B^{(j)}[j+1]\end{pmatrix}\\
        = & - \frac{a_0 -z}{a_{-r}} (-1)^{r+1} \det \begin{pmatrix}B^{(j)}[j+1:j+r]\end{pmatrix}.
    \end{align*}
    By hypothesis \ref{fourhyp}$\phantom{}_j$, we know that the polynomial $\det \begin{pmatrix}B^{(j)}[j+1:j+r]\end{pmatrix}$ is of degree $j$, then the polynomial $- \frac{a_0 -z}{a_{-r}} (-1)^{r+1} \det \begin{pmatrix}B^{(j)}[j+1:j+r]\end{pmatrix}$ is of degree $j+1$ and \ref{fourhyp}$\phantom{}_{j+1}$ follows.

    \emph{Conclusion:} the matrix $B^{(N-r)}$ gives the result, where \[C(z) \egdef B^{(N-r)}[1:r,N-r+1 : N].\]

    \textbf{Proof of uniqueness:} assume that $C$ and $C'$ are satisfying the lemma. Then 
    $$B \sim_{\G(z)} \underbrace{\begin{pmatrix} \begin{array}{ccc|c} 0& \cdots& 0& \\ \vdots & &\vdots & \hspace{3mm} C(z)\\ 0& \cdots &0 & \end{array}\end{pmatrix}}_{=\nomCa} \sim_{\G(z)} \underbrace{\begin{pmatrix} \begin{array}{ccc|c} 0& \cdots& 0& \\ \vdots & &\vdots & \hspace{3mm} C'(z) \\ 0& \cdots &0 & \end{array}\end{pmatrix}}_{=\nomCb}.$$

    On the one side, we have $(\nomCa - \nomCb)\project_{|\G(z)} = 0$ and on the other side, because the $N-r$ first columns are zero, we have
    $(\nomCa-\nomCb)_{|\text{Vect}(e_1,\dots, e_{N-r})} = 0$ where $e_1, \dots, e_N$ is the canonical basis of $\C^N$.

	Let us introduce the linear subspace $F \egdef \ker A$ where 
	\[A \egdef \begin{pmatrix}a_{-r} & \dots & (a_0-z) &  &  0 \\ & \ddots &   & \ddots & \\ 0 &  & a_{-r} & \dots & (a_0-z) \end{pmatrix}\in \mathcal M_{N-r, N}(\C).\]

	We have $F\cap \text{Vect}(e_1,\dots, e_{N-r}) = \{0\}$.
		Indeed if $x \in F\cap \text{Vect}(e_1,\dots, e_{N-r})$, then $Ax = 0$ and $x = (x_1,\dots, x_{N-r},0, \dots, 0)^T$. By solving the triangular system
		\[\begin{cases}
			a_{-r}x_1 + a_{-r+1} x_2 + \cdots +a_{-1}x_r + (a_0 - z) x_{r+1} = 0 \\
			\vdots\\
			a_{-r}x_{N-r -1} + a_{-r+1}  x_{N-r} = 0 \\
			a_{-r}  x_{N-r}=0,
		\end{cases}\]
	we find $x = 0$. Moreover, $\dim F = r$ by rank-nullity theorem, then we have $$F\oplus \text{Vect}(e_1,\dots, e_{N-r}) =\C^N.$$

	We want to show $(\nomCa - \nomCb)_{|F} = 0$ and we know that $(\nomCa- \nomCb)\project_{|\G} = 0$. Let $x \in F$ and extend it to $\tilde x \in \G$. To that aim, it suffices to set recursively for all $j > N$, $$\tilde x_{j} = \dfrac{1}{a_0-z}(-a_{-1}\tilde x_{j-1} - \cdots - a_{-r}\tilde x_{j-r}).$$ It follows that $(\nomCa - \nomCb)\project(\tilde x) = 0$ and  $(\nomCa - \nomCb)x = 0$. Then, we have $(\nomCa - \nomCb) = 0$ on $\C^N$.
\end{proof}

\begin{rem}
    The uniqueness result is actually not needed for the next results.
\end{rem}

In Section~\ref{sec:Upwindnum} (resp. Section~\ref{sec:BWexamplenumerical}), we perform the explicit algorithmic computations described above for the classic first-order upwind scheme~\eqref{eq:upwindint} (resp. Beam-Warming scheme~\eqref{eq:BW}).


\subsection{Holomorphy} 

\begin{lemma}\label{lem:detKLquotient}
    Assume $|a_0|<1$. For all $\ell \in \N$ and for all $z \in \overline{\U}$, we have 
    \begin{equation}\label{eq:fctn}\dfrac{\det{K_{\ell, \ell+r-1}(z)}}{\det{K_{0,r-1}(z)}} = (-1)^{\ell r} \left (\dfrac{a_{-r}}{a_0-z}\right )^\ell.\end{equation}
\end{lemma}

\begin{proof}
    \textbf{Case with only one root $\kappa$ of multiplicity $\mult$.}
    By Lemma~\ref{lem:hersh} (Hersh), we know that $\mult = r$, but let keep $\mult$ because it will be useful for the next step. 
    We recall that \begin{equation}\label{eq:detK0}\det K_{0,r-1} = \begin{vmatrix} 1 &  0 & \cdots & 0 \\ \kappa & \kappa & \cdots & \kappa\\ \kappa^2 & 2\kappa^2 & \cdots & 2^{\mult-1}\kappa^{2} \\ \vdots & & & \vdots \\ \kappa^{r-1} & (
            r-1)\kappa^{r-1} & \cdots & (r-1)^{\mult-1} \kappa^{r-1} \end{vmatrix}.\end{equation}
    We want to work on
    \begin{align*}\det K_{\ell, \ell+r-1} & = \begin{vmatrix} \kappa^\ell & \ell \kappa^\ell & \cdots & \ell^{\mult-1} \kappa^\ell \\ \kappa^{\ell+1} & (
        \ell +1)\kappa^{\ell+1} & \cdots & (\ell+1)^{\mult-1} \kappa^{\ell+1} \\ \vdots & & & \vdots \\ \kappa^{\ell+r-1} & (
            \ell +r-1)\kappa^{\ell+r-1} & \cdots & (\ell+r-1)^{\mult-1} \kappa^{\ell+r-1} \end{vmatrix}\\
            & = \kappa^{\ell\mult}\begin{vmatrix} 1 & \ell  & \cdots & \ell^{\mult-1}  \\ \kappa & (
                \ell +1)\kappa & \cdots & (\ell+1)^{\mult-1} \kappa\\ \vdots & & & \vdots \\ \kappa^{r-1} & (
                    \ell +r-1)\kappa^{r-1} & \cdots & (\ell+r-1)^{\mult-1} \kappa^{r-1} \end{vmatrix}.
    \end{align*}

    We do some operations on columns to recover \eqref{eq:detK0}. For $n$ from $\mult -1$ to $0$, we replace the column $C_n$ by  $\sum_{k=0}^n (-\ell)^{n-k} \binom{n}{k} C_k$. After the transformation, the component in position $(i,n)$, with $i\in \interval{0}{r-1}$ and $n \in \interval{0}{\mult-1}$, is 
    \begin{align*}
        \sum_{k = 0}^n (-\ell)^{n-k} \binom{n}{k} (\ell+i)^k\kappa^i & = \sum_{k = 0}^n (-\ell)^{n-k} \binom{n}{k} \sum_{s= 0}^{k} \binom{k}{s} \ell^{k-s} i^{s}\kappa^i\\
        & = \sum_{k = 0}^n \sum_{s= 0}^{k} (-\ell)^{n-k} \binom{n}{s}  \binom{n-s}{k-s} \ell^{k-s} i^{s}\kappa^i \\
        & =\sum_{s = 0}^n \binom{n}{s}i^{s}\kappa^i  \sum_{k= s}^{n} (-\ell)^{n-k}   \binom{n-s}{k-s} \ell^{k-s}\\
        & = \sum_{s = 0}^n \binom{n}{s}i^{s}\kappa^i  \underbrace{\sum_{\tilde k= 0}^{n-s} (-\ell)^{n-s-\tilde k}   \binom{n-s}{\tilde k} \ell^{\tilde k}}_{= \delta_{n,s}}= i^{n}\kappa^i.
    \end{align*}
    This is exactly the component in $(i,n)$ of the  matrix $K_{0,r-1}(z)$.
    
    \textbf{General case.}
    We can do the same operation on columns for each root. We take out $\kappa_1^{\ell\mult_1} \cdots \kappa_\nbrmult^{\ell \mult_\nbrmult}$, and for each root $\kappa_j$ with $j \in \interval{1}{M}$, we vary $n_{\kappa_j}$ from $\mult_j -1$ to $0$ and modify columns linked to $\kappa_j$. We regain matrix $K_{0,r-1}(z)$.

    \textbf{Conclusion.}

    We proved $$\dfrac{\det K_{\ell,\ell+r-1}}{\det K_{0,r-1}} = \kappa_1^{\ell\mult_1} \cdots \kappa_\nbrmult^{\ell \mult_\nbrmult}.$$
    Observe that $a_0-z\neq 0$ because $z\in\overline{\U}$ and $a_0\in\D$. Therefore, by Vieta's formulas for the polynomial~\eqref{eq:eqcharac}, we finally have 
    \[\kappa_1^{\mult_1} \cdots \kappa_\nbrmult^{\mult_\nbrmult} = (-1)^r \dfrac{a_{-r}}{a_0-z}.\]
\end{proof}

Lemma~\ref{lem:detKLquotient} implies the holomorphy on $\U$ and continuity on $\overline{\U}$ of the function in~\eqref{eq:fctn}.

For the sake of completeness in the forthcoming proofs, we state hereafter two elementary lemmas. Both are easily deduced from classic properties of the winding number in complex analysis (see~\cite{Lang13}).
\begin{lemma}\label{lem:merompole}
    Let $P$ and $Q$ be two polynomials with $\deg P> \deg Q$. If the function $z \mapsto P(z)Q(z)^{-1}$ is holomorphic on $\U$ then $z \mapsto P(1/z)Q(1/z)^{-1}$ is meromorphic on $\mathbb D$ with only one pole, at the origin of order $\deg P - \deg Q$.
\end{lemma}

\begin{lemma}\label{lem:inverseindice}
    Let $f$ be a holomorphic function on $\U$ and continuous on $\overline{\U}$ and $g$ be the function defined on $\overline{\mathbb D}^*$ by $g : z \mapsto f(1/z)$. Then, one has $\mathrm{Ind}_{g(\S)}(0) = - \mathrm{Ind}_{f(\S)}(0)$.
\end{lemma}


\subsection{Explicit form of the intrinsic Kreiss-Lopatinskii determinant}

In the previous Lemmas~\ref{lem:algebra} and~\ref{lem:detKLquotient}, the assumption $|a_0|<1$ is made.  This is not a restriction since this is a consequence of the supplemented consistency assumption.

\begin{lemma}\label{lemcoeffa0}
    Let the scheme~\eqref{eqprincip} be Cauchy-stable \ref{assumption:cauchystab} and consistent \ref{assumption:consistency}, then $|a_0|<1$.
\end{lemma}

\begin{proof}
    We have 
    \[a_0 = \dfrac{1}{2\pi} \int_0^{2\pi} \sum_{k=-r}^p a_k e^{ik\xi} d\xi.\]
    Integrating on the unit circle, by triangle inequality and Cauchy-stability, we have \[|a_0| \< \dfrac{1}{2\pi} \int_0^{2\pi} \underbrace{\left |\sum_{k=-r}^p a_k e^{ik\xi}\right |}_{\< 1} d\xi\< 1.\]
    Let us assume now the identity $|a_0| = 1$, so that the equality occurs within the previous triangle inequality. Therefore there exists a real-valued function $g$ and a complex $\alpha$ such that for all $\xi\in \R$, we have $\sum_{k = -r}^p a_k e^{ik\xi} = \alpha g(\xi)$. Now at the point $\xi = 0$, we obtain
    $1 = \sum_{k = -r}^p a_k = \alpha g(0)$. Therefore $\alpha$ is real, as well as the symbol $\gamma(\xi)$. Using the complex conjugate we deduce $\sum_{k = -r}^p a_k e^{ik\xi} - \sum_{k=-p}^r a_{-k}e^{ik\xi}=0$ and then by the injectivity of the Fourier coefficients, it follows that $p = r$ and $a_k = a_{-k}$ for all $k \in \{1,\dots,r\}$. Finally, using now the consistency assumption, one has:
    \[ 0 = \sum_{k=-r}^p k a_k = - \lambda \neq 0.\] 
    By this contradiction, the proof is complete.
\end{proof}

Now every piece can be put together to prove  Theorem~\ref{thrm:rationalfraction}.

\begin{proof}[Proof of \Cref{thrm:rationalfraction}]
    Let us recall the function 
    \[\fun{\DKLindep}{z}{\overline{\U}}{\dfrac{\det BK_{-r,m-1}(z)}{\det K_{0,r-1}(z)}}{\C}.\]
    By \Cref{lemcoeffa0}, we will be able to use \Cref{lem:algebra} and  \Cref{lem:detKLquotient}.
    With \Cref{lem:algebra} and \Cref{rem:BK=CK}, we express $\DKLindep$ as 
    \begin{equation}\label{eq:calculproofDelta1}\DKLindep(z) = \dfrac{\det C(z) \det K_{m-r,m-1}(z)}{\det K_{0,r-1}(z)}\end{equation}
    where $C(z)$ is polynomial with respect to $z$.

    By \Cref{lem:detKLquotient}, we have 
    \begin{equation}\label{eq:calculproofDelta2}
        \dfrac{\det K_{m-r,m-1}(z)}{\det K_{0,r-1}(z)} = (-1)^{r(m-r)} \left (\frac{a_{-r}}{a_0-z}\right )^{m-r}.
    \end{equation}

    By \Cref{lemcoeffa0}, $a_0$ cannot be a pole of $\DKLindep$, then the function $\DKLindep$ can be written, for $z \in \overline{\U}$,  as

    \begin{equation}\label{eq:DetKLindep} \DKLindep(z) = (-1)^{r(m-r)}\det C(z)\left (\dfrac{a_{-r}}{a_0 - z}\right )^{m-r}, \end{equation}
    where $\det C(z)$ is a polynomial of $z$ and $(a_0 - z)$ does not vanish because $z\in \overline{\U}$ and $a_0 \in \D$. 
\end{proof}

    The proof of \Cref{thrm:nbrzerodet} relies on the residue theorem to count the zeros of a holomorphic function.

\begin{proof}[Proof of \Cref{thrm:nbrzerodet}]
    By \Cref{thrm:rationalfraction}, the function $\DKLindep$ is holomorphic on $\U$ and continuous on $\overline{\U}$.

    Let take the function
    \[\fun{\DKLindepinvers}{z}{\D^*}{\DKLindep(1/z)}{\C.}\]
    \textbf{The function $\DKLindepinvers$ is meromorphic on $\D$ with a pole in 0 of order $r$.}

    By \Cref{lem:algebra}, we have $\deg \det C(z) = m$ because the $r$ first columns of $B$ form the identity matrix of size $r$ which is invertible.
    By \Cref{lem:merompole} with $P = \det C(z)$ and $Q = (-1)^{r(m-r)}  \frac{(a_0-z)^{m-r}}{a_{-r}^{m-r}}$, the only pole of $\DKLindepinvers$ is in 0 and of order $\deg\det C(z) - (m-r) = m - (m-r) = r$.

    \textbf{Residue theorem on $\DKLindepinvers$}
    The function $\DKLindep$ is continuous on $\overline{\U}$, then the function $\DKLindepinvers$ is continuous on $\overline{\D}^*$. We can use the residue theorem on $\DKLindepinvers$ with the unit circle $\mathbb S$ as loop around 0. Then we have
    \[ \mathrm{Ind}_{\DKLindepinvers(\S)}(0) = \# \text{zeros}_{\DKLindepinvers}(\D) - \# \text{poles}_{\DKLindepinvers}(\D).\]

    \textbf{Conclusion}
    We have $\#\text{zeros}_{\DKLindep}(\U) = \#\text{zeros}_{\DKLindepinvers}(\D)$
    and, by \Cref{lem:inverseindice}, we have $\mathrm{Ind}_{\DKLindepinvers(\S)}(0) = - \mathrm{Ind}_{\DKLindep(\S)}(0)$. It follows that
    \[\#\text{zeros}_{\DKLindep}(\U) = \underbrace{\# \text{poles}_{\DKLindepinvers}(\D)}_{r}- \mathrm{Ind}_{\DKLindep(\S)}(0).\]

This concludes the proof.
\end{proof}


\section{Numerical results}\label{sec:numerical}

In this section, we first explain the numerical computation of the winding number of the origin in order to use Corollary~\ref{thrm:nbrzerodet} and Method~\ref{proc:numericalprocedure}. 
The simplest first order upwind scheme is then quickly treated, but for a general three-points boundary condition. Next, a main class of high-order boundary conditions, known as the simplified Lax-Wendroff procedure, is presented. They will be used together with the Beam-Warming scheme. After introducing the Beam-Warming scheme, we present computations of Kreiss-Lopatinskii determinant and numerical illustrations.
Finally, we study the stability of discretizations where the physical boundaries are not aligned with the mesh.

\subsection{Computation of the winding number}\label{sec:windingnbr}

In the forthcoming numerical illustrations, the interest of Method~\ref{proc:numericalprocedure} is showcased. Indeed, \Cref{thrm:nbrzerodet} makes the link between the number of zeros of a holomorphic function and the winding number of a curve which is easy to compute. In fact, as an integer is expected, the approximation of the winding number is generally more reliable, contrary to a real or complex computation because of machine precision.

When the origin is not on the curve, there are different ways to compute the winding number of the origin with respect to the curve. Either we can apply the definition and compute approximately a complex integral, or we can count the number of paths around the origin by using a polygonal approximation of the curve. 
The second approach is studied by Garcia-Zapata and Martin \cite{Zapata12,Zapata14} with a careful numerical treatment that consists in detecting the possible proximity of the curve to the origin.
To that aim, the discretization of the curve is locally refined by an so-called "insertion procedure with control of singularity".
Indeed, by the explicit formula \eqref{eq:explicitformula} of the intrinsic Kreiss-Lopatinskii determinant, the curve $\DKLindep(\S)$ is clearly parameterized by the lipschitz function $\DKLindep$, thus satisfies the required assumptions from the result in \cite{Zapata12} and \cite{Zapata14}.


\subsection{Upwind scheme}\label{sec:Upwindnum}

The easiest example of totally upwind scheme is the usual first-order upwind scheme defined, for $j\in \N$ and $n \in \N$, by
\begin{equation}\label{eq:upwindint}
    U_j^{n+1} = \lambda U_{j-1}^n + (1-\lambda) U_j^n
\end{equation}

and supplemented at the boundary, for example, by
\begin{equation}\label{eq:upwindbord}
    U_{-1}^n = b_0 U_0^n + b_1 U_1^n + b_2 U_2^n
\end{equation}
with arbitrary coefficients $b_0$, $b_1$ and $b_2$.
For that scheme, we have $r=1$, $m=3$ and the characteristic equation \eqref{eq:eqcharac} reads \begin{equation}\label{eq:eqcharacupwind}z\kappa(z) = \lambda + (1-\lambda)\kappa(z).\end{equation} 
The scheme is Cauchy-stable for $\lambda \in ]0,1]$ and one can check that for $0<\lambda \<1$ and for $|z|>1$, the root $\kappa(z)$ of \eqref{eq:eqcharacupwind} is in $\D$ by Lemma~\ref{lem:hersh} (Hersh).

Let us now execute the computation of the intrinsic Kreiss-Lopatinskii determinant, as presented in~Lemma~\ref{lem:algebra} (here $N=4$):

\begin{align*}
    & B^{(0)} = \begin{pmatrix} 1 & -b_0 & -b_1& -b_2 \end{pmatrix}\\
    \leadsto & B^{(1)} = \begin{pmatrix} 0 & -b_0-\frac{1-\lambda-z}{\lambda} & -b_1 & -b_2\end{pmatrix} \\
    \leadsto & B^{(2)} = \begin{pmatrix} 0 & 0 & -b_1 + (b_0+\frac{1-\lambda-z}{\lambda})\frac{1-\lambda-z}{\lambda} & -b_2 \end{pmatrix} \\
    \leadsto & B^{(3)} = \begin{pmatrix} 0 & 0 & 0& -b_2 + (b_1 - (b_0+\frac{1-\lambda-z}{\lambda})\frac{1-\lambda-z}{\lambda})\frac{1-\lambda-z}{\lambda} \end{pmatrix} 
\end{align*}

It follows that $\det C(z) = -b_2 + (b_1 - (b_0+\frac{1-\lambda-z}{\lambda})\frac{1-\lambda-z}{\lambda})\frac{1-\lambda-z}{\lambda}$. 

Hence, the explicit formula~\eqref{eq:explicitformula} reads as follows:
\begin{equation*}
    \DKLindep(z)  = (-1)^{2}\det C(z)\left (\dfrac{a_{-r}}{a_0 - z}\right )^{2} = - \dfrac{b_2 \lambda^2}{(1-\lambda - z)^2} + \dfrac{b_1\lambda}{1-\lambda -z} - b_0 - \dfrac{1-\lambda - z}{\lambda}.
\end{equation*}

A similar computation can be achieved for boundary conditions with larger $m$ and/or for totally upwind schemes with a larger stencil (see below for the Beam-Warming scheme).


\subsection{Simplified inverse Lax-Wendroff procedure}\label{sec:SILWprocedure}

As explained in \cite{Tan10} and \cite{Vilar15}, the inverse Lax-Wendroff procedure is used to improve the consistency at the boundary by using the PDE to transform space derivative into time derivative. Namely, for the advection equation \eqref{eq:advection}, the following relation holds, for $k \in \N^*$,
\[\dfrac{\partial^k u}{\partial x^k}  = \dfrac{(-1)^k}{a^k} \dfrac{\partial^k u}{\partial t^k}.\]
By a Taylor expansion at order $d$ to approximate $u(n\dt,j\dx)$ for $n\in \N$ and $j\in \interval{-r}{-1}$, one can then define the ghost points used in the boundary condition \eqref{eqbord} by
\begin{equation*}
    U_{j}^n  = \sum_{k=0}^{d-1} \dfrac{(j\dx)^k}{k!} \dfrac{\partial^ku}{\partial x^k}(n\dt,0)  = \sum_{k=0}^{d-1} \dfrac{(j\dx)^k}{k!} (-1)^k\dfrac{g^{(k)}(n\dt)}{a^k}.
\end{equation*}
However, many derivatives of the datum $g$ are requireobtain a high order approximation and the complexity then severely increases for multidimensional situations. As explained in \cite{Vilar15}, the simplified inverse Lax-Wendroff procedure of order $d$ with simplified order $k_d$ that we call “S$k_d$ILW$d$” may be used when derivatives of $g$ are not known. Therefore, the first $k_d-1$ derivatives of $g$ are considered and then for the next terms between order $k_d$ and $d$, an extrapolation procedure is used. Finally the general formula is, for $j \in \interval{-r}{-1}$, the following one
\begin{equation}\label{eq:SILWproc}
    U_{j}^n = \sum_{k=0}^{k_d-1} \dfrac{(-j\dx)^k}{k!} \dfrac{g^{(k)}(n\dt)}{a^k}+ \sum_{k=k_d}^{d-1} \dfrac{j^k}{k!} \sum_{s=0}^{k} \binom{k}{s} (-1)^{k-s} U_{s}^n.
\end{equation}


\subsection{Beam-Warming scheme}

The Beam-Warming scheme with simplified inverse Lax-Wendroff of order 3 and simplified order 2 reads
\begin{equation}\label{eq:BW}\begin{cases}
    U_{j}^{n+1} = \dfrac{\lambda 	(\lambda-1)}{2} U_{j-2}^n + \lambda (2 - \lambda) U_{j-1}^n + \dfrac{(\lambda -1)(\lambda -2)}{2}U_j^n,\\
    U_{-1}^n = g(t^n) + \dfrac{\dx g'(t^n)}{a}+ \dfrac{1}{2}(U_2^n-2U_1^n+U_0^n)
    ,  \\
    U_{-2}^n = g(t^n) + \dfrac{2\dx g'(t^n)}{a}+ 2(U_2^n-2U_1^n+U_0^n)
    ,\\
    U_j^0 = 0.
\end{cases}\end{equation}

This scheme satisfies Assumptions~\ref{assumption:totallyupwind} and~\ref{assumption:consistency}. To have the Cauchy-stability assumption~\ref{assumption:cauchystab}, we study the symbol with respect to the CFL condition $\lambda$. From \eqref{eq:BW}, the symbol is 
\[\symbole(\xi) = \frac{\lambda(\lambda-1)}{2} e^{-2i\xi} + \lambda (2 - \lambda) e^{-i\xi} + \frac{(\lambda -1)(\lambda -2)}{2}.\]
In the Figure~\ref{fig:BWsymbol}, this symbol is represented for $\lambda = 1.8$.
     
\begin{figure}[!ht]
    \centering
    \includegraphics[trim = 2cm 1.4cm 2.5cm 2.1cm, clip, width = 10cm]{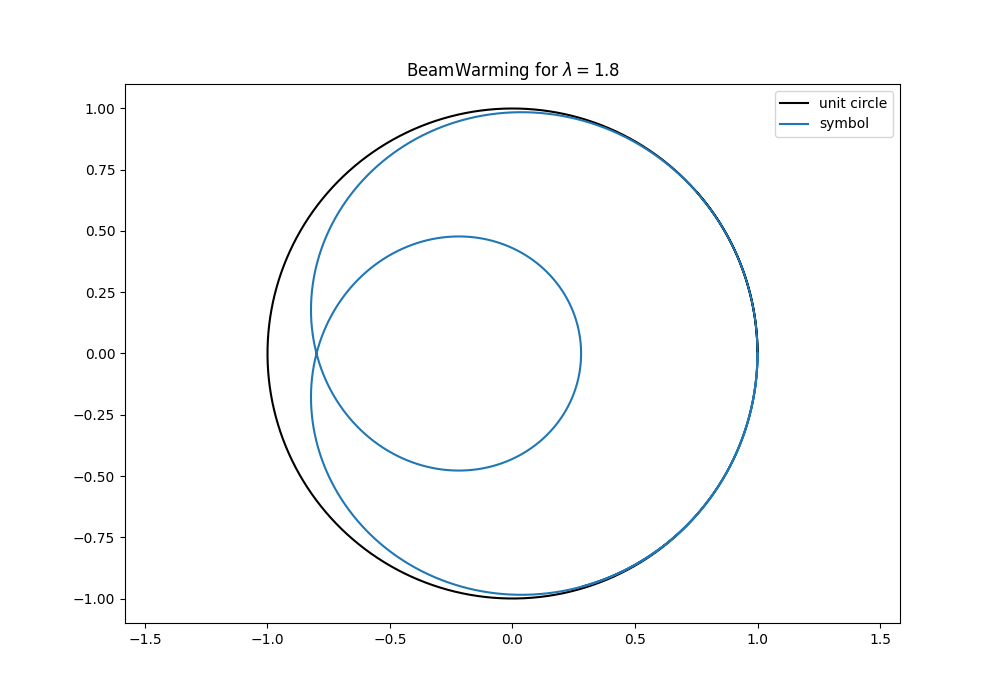}
    \caption{Symbol of Beam-Warming scheme for $\lambda = 1.8$.}\label{fig:BWsymbol}
\end{figure}

\begin{prop}
    The Beam-Warming scheme is Cauchy-stable if and only if 
    $0<\lambda\<2.$
\end{prop}

Even if it is a classic result, we recall  the outline of the proof. 
\begin{proof}
    While computing the symbol, we have, for all $\xi \in \R$,
    \begin{equation*}
        \symbole(\xi)  = \dfrac{(\lambda -1)(\lambda -2)}{2} + \lambda (2 - \lambda)e^{-i\xi} + \dfrac{\lambda(\lambda-1)}{2} e^{-2i\xi} 
         = e^{-i\xi}\left (\lambda(\lambda-1)\cos\xi + \lambda(2-\lambda) - (\lambda-1)e^{i\xi} \right ).
    \end{equation*}
    Thus, the modulus of the symbol is after some easy computations
    \begin{equation*}
        |\symbole(\xi)|^2  = 1 - \lambda(2-\lambda)(\lambda-1)^2(1-\cos \xi)^2.
    \end{equation*}

    To be Cauchy-stable, we must have $|\symbole(\xi)|^2 \< 1$, so we want to have $\lambda(2- \lambda)(\lambda-1)^2 (1-\cos \xi)^2 \>0$. Because $\lambda>0$, then the condition is $\lambda\<2$.
\end{proof}

The non-degeneracy assumption \ref{assumption:nondegenerate} is related to the value $r = 2$ for $\lambda \in ]0,2]\setminus\{1\}$ and to the value $r=1$ for $\lambda = 1$.
This example will be useful to illustrate the theory, especially in the following subsection.


\subsection{Kreiss-Lopatinskii determinant computation for Beam-Warming scheme}\label{sec:BWexamplenumerical}
First, we compute the Kreiss-Lopatinskii determinant $\DKL$ from Definition~\ref{def:detKL} for the Beam-Warming scheme with S2ILW3 boundary condition as in \eqref{eq:BW}. Assuming that the the roots of \eqref{eq:eqcharac} are distinct for a given $|z|\> 1$, we have
    \[\DKL(z) = \det \begin{pmatrix} \kappa_1^{-2} -2 + 4\kappa_1 -2 \kappa_1^2 & \kappa_2^{-2} -2 + 4\kappa_2 - 2 \kappa_2^2 \\ \kappa_1^{-1} - \frac{1}{2} +\kappa_1 - \frac{\kappa_1^2}{2} & \kappa_2^{-1} - \frac{1}{2} +\kappa_2 - \frac{\kappa_2^2}{2}\end{pmatrix}.\]
    If there is one single root with multiplicity 2, then we have 
    \[\DKL(z) = \det \begin{pmatrix} \kappa_1^{-2} -2 + 4\kappa_1 -2 \kappa_1^2 & -2\kappa_1^{-2}  + 4\kappa_1 - 4 \kappa_1^2 \\ \kappa_1^{-1} - \frac{1}{2} +\kappa_1 - \frac{\kappa_1^2}{2} & -\kappa_1^{-1}  +\kappa_1 - \kappa_1^2\end{pmatrix}.\]

In the rest of this section, we continue the example of the Beam-Warming scheme \eqref{eq:BW} so as to illustrate practically the algebraic transformation set up in \Cref{lem:algebra}.

For that scheme, the corresponding $\mathcal Z$-transformed equation \eqref{eq:Ztransform} is, for $j\in \N$,
$$z\widetilde{U}_j(z) = a_{-2}\widetilde{U}_{j-2}(z) + a_{-1}\widetilde{U}_{j-1}(z) + a_0\widetilde{U}_j(z),$$
involving the coefficients $a_0 = \frac{(\lambda - 1)(\lambda -2)}{2}$, $a_{-1} = \lambda (2-\lambda)$ and $a_{-2} = \frac{\lambda (\lambda-1)}{2}$.

Let us denote in the following lines $\alpha \egdef \frac{-a_{-1}}{a_{-2}}$ and $\beta \egdef \frac{z-a_0}{a_{-2}}$ so that the linear recurrence relation has now, for $j \in \N$, the form below:
\begin{equation}\label{eq:recurrBW}
    \widetilde{U}_{j-2}(z) = \alpha \widetilde{U}_{j-1}(z) + \beta\widetilde{U}_{j}(z).
\end{equation}

The considered boundary condition involves the following matrix:
\[B = \begin{pmatrix} 1 & 0 & -2& 4&-2 \\ 0 & 1 & -\tfrac 1 2 &1 & -\tfrac 1 2 \end{pmatrix}\]

with dimensions $r= 2$ and $N = 5$. With the notations in the proof of \Cref{lem:algebra}, let us now construct the matrix $C(z) = B^{(3)}[1:2,4:5]$. To that aim, we transform successively the matrix $B$ so as to keep unchanged the vector $B \hspace{-0.5mm}\left (\widetilde{U}_{j-2}(z)\widetilde{U}_{j-1}(z) \widetilde{U}_{j}(z)\widetilde{U}_{j+1}(z)\widetilde{U}_{j+2}(z)\hspace{-0.5mm}\right )^T$ thanks to the recurrence relation~\eqref{eq:recurrBW}. Hereafter are the steps:
\begin{align*}
    & B^{(0)} = \begin{pmatrix} 1 & 0 & -2& 4& -2 \\ 0 & 1 & -\frac{1}{2}&1 & -\frac{1}{2} \end{pmatrix}\\
    \leadsto & B^{(1)} = \begin{pmatrix} 0 & \alpha & -2+\beta& 4&-2 \\ 0 & 1 & -\frac{1}{2} &1 & -\frac{1}{2} \end{pmatrix} \\
    \leadsto & B^{(2)} = \begin{pmatrix} 0 & 0 & -2+\beta + \alpha^2& 4 +\alpha \beta&-2 \\ 0 & 0 & -\frac{1}{2} + \alpha &1 + \beta &-\frac{1}{2} \end{pmatrix} \\
    \leadsto & B^{(3)} = \begin{pmatrix} 0 & 0 & 0& 4  +\alpha \beta + \alpha(-2+\beta + \alpha^2)&-2 +\beta(-2+\beta +\alpha^2) \\ 0 & 0 & 0 &1 + \beta +\alpha(-\frac{1}{2} + \alpha) & - \frac{1}{2} + \beta(-\frac{1}{2} + \alpha) \end{pmatrix} 
\end{align*}
From there, it follows that
\[
    C(z) = \begin{pmatrix} 4  +\alpha \beta + \alpha(-2+\beta + \alpha^2)&-2 +\beta(-2+\beta +\alpha^2) \\ 1 + \beta +\alpha(-\frac{1}{2} + \alpha) & - \frac{1}{2} + \beta(-\frac{1}{2} + \alpha) \end{pmatrix},
\]
and thus
 \begin{align*}
    \det C(z) & =  (4  +\alpha \beta + \alpha(-2+\beta + \alpha^2))(- \tfrac{1}{2} + \beta(-\tfrac{1}{2} + \alpha)) \\& \hspace{2cm} - (1 + \beta +\alpha(-\tfrac{1}{2} + \alpha))(-2 +\beta(-2+\beta +\alpha^2))
    \\ & = -\beta^3 + \beta^2 + 2\beta - \alpha \beta^2/2 + 3\alpha \beta  - \alpha^2\beta -  2\alpha^2 - \alpha^3/2.
\end{align*}

The intrinsic Kreiss-Lopatinskii determinant explicit formula \eqref{eq:DetKLindep} (with here $m=3$ and $r=2$) is the following:
\[\DKLindep(z) = \dfrac{-1}{\beta}(-\beta^3 + \beta^2 + 2\beta - \alpha \beta^2/2 + 3\alpha \beta  - \alpha^2\beta -  2\alpha^2 - \alpha^3/2).\]

On Figure~\ref{fig:DetKLBWS2ILW3}, the curve $\DKLindep(\mathbb S)$ is represented successively for different values of the CFL parameter~$\lambda$. The goal is to compute the winding number of $0$, concerned with Corollary~\ref{thrm:nbrzerodet} in order to tackle stability thanks to the Theorem~\ref{thrmKreiss} (Kreiss). A premultiplication of the quantity $\DKLindep$ by $a_{-2}^{2}$ may reduce the order of magnitude of the curves, without changing the winding number. The left and right figures correspond to the case with or without rescaling.

By \Cref{thrm:nbrzerodet}, we have $\# \mathrm{zeros}_{\DKLindep} = r - \Ind_{\DKLindep(\S)}(0)$ but after dividing $\DKLindep$ by $z^r$:
$$\DKLindepdivided : z \mapsto \DKLindep(z) / z^r,$$
we obtain $\# \mathrm{zeros}_{\DKLindep} = - \Ind_{\DKLindepdivided(\S)}(0)$, because $\Ind_{\DKLindepdivided(\S)}(0) = \Ind_{\DKLindep(\S)}(0) - r$, see Figure~\ref{fig:dividedDetKLBWS2ILW3}.

\begin{figure}

    \begin{minipage}[b]{0.48\linewidth}
     \centering
     \includegraphics[trim = 2.07cm 1.7cm 2.51cm 2.41cm, clip, width = 7cm]{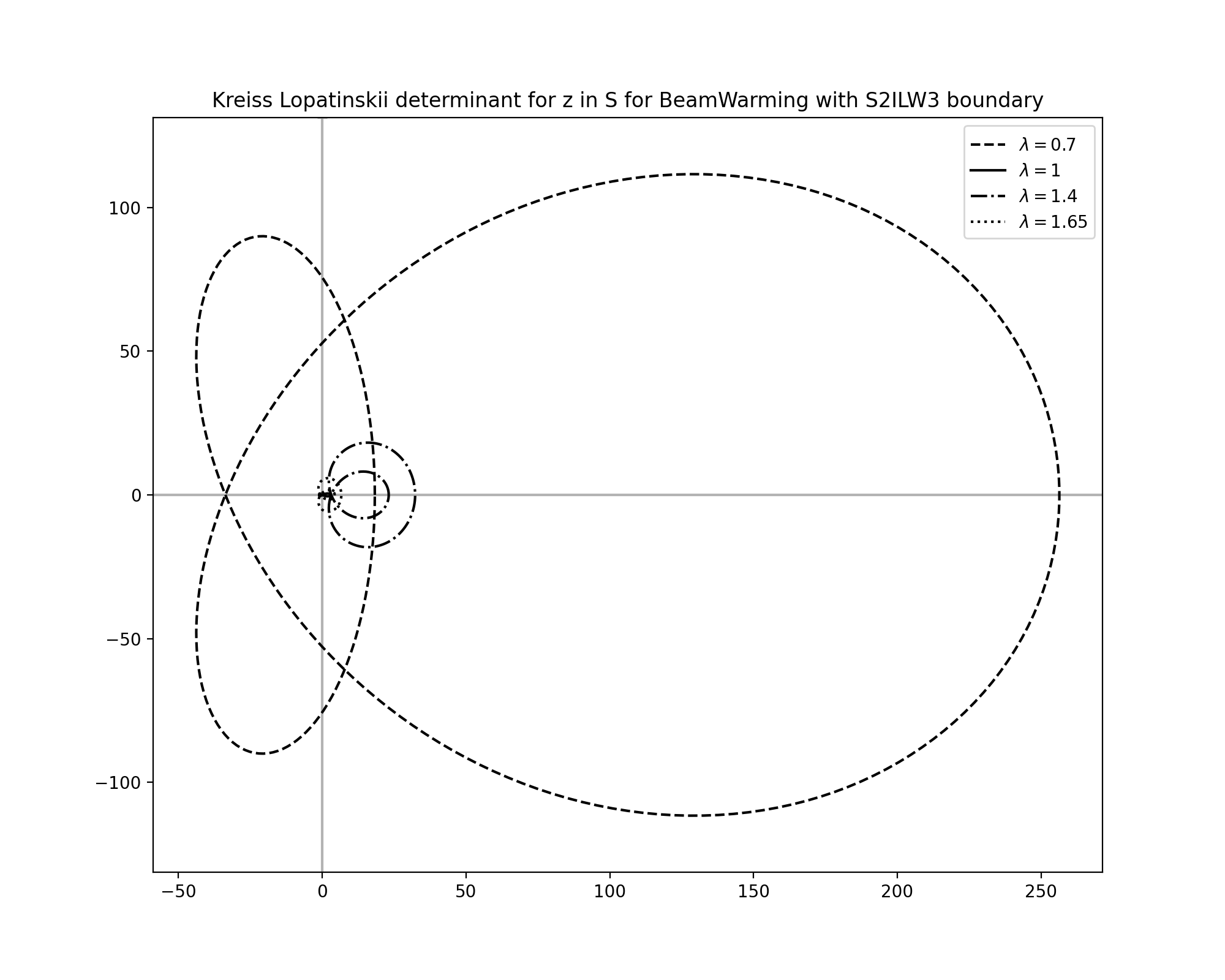}    
    \end{minipage}
  \hfill
    \begin{minipage}[b]{0.48\linewidth}
     \centering
     \includegraphics[trim = 2.12cm 1.7cm 2.51cm 2.41cm, clip, width = 7cm]{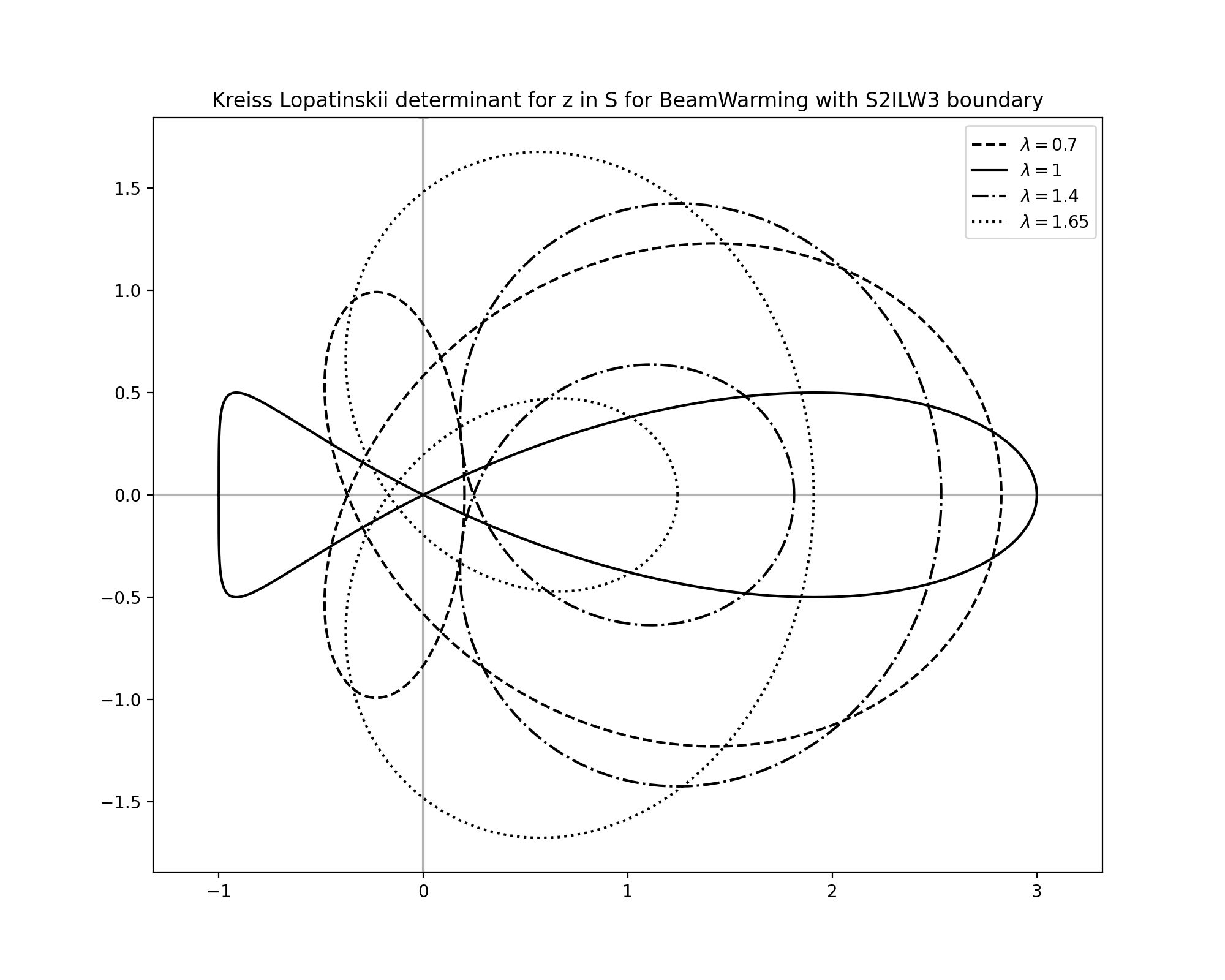}   
    \end{minipage}
  
    \caption{Kreiss-Lopatinskii Determinant $\DKLindep$ when $z$ is on $\S$ for scheme \eqref{eq:BW} for $\lambda \in \{0.7,1,1.4,1.65\}$ (left) and the rescaled one $a_{-2}^2 \DKLindep$ (right).}
    \label{fig:DetKLBWS2ILW3}
  
  \end{figure}

\begin{figure}
    \centering
    \includegraphics[trim = 2.12cm 1.7cm 2.51cm 2.41cm, clip, width = 12cm]{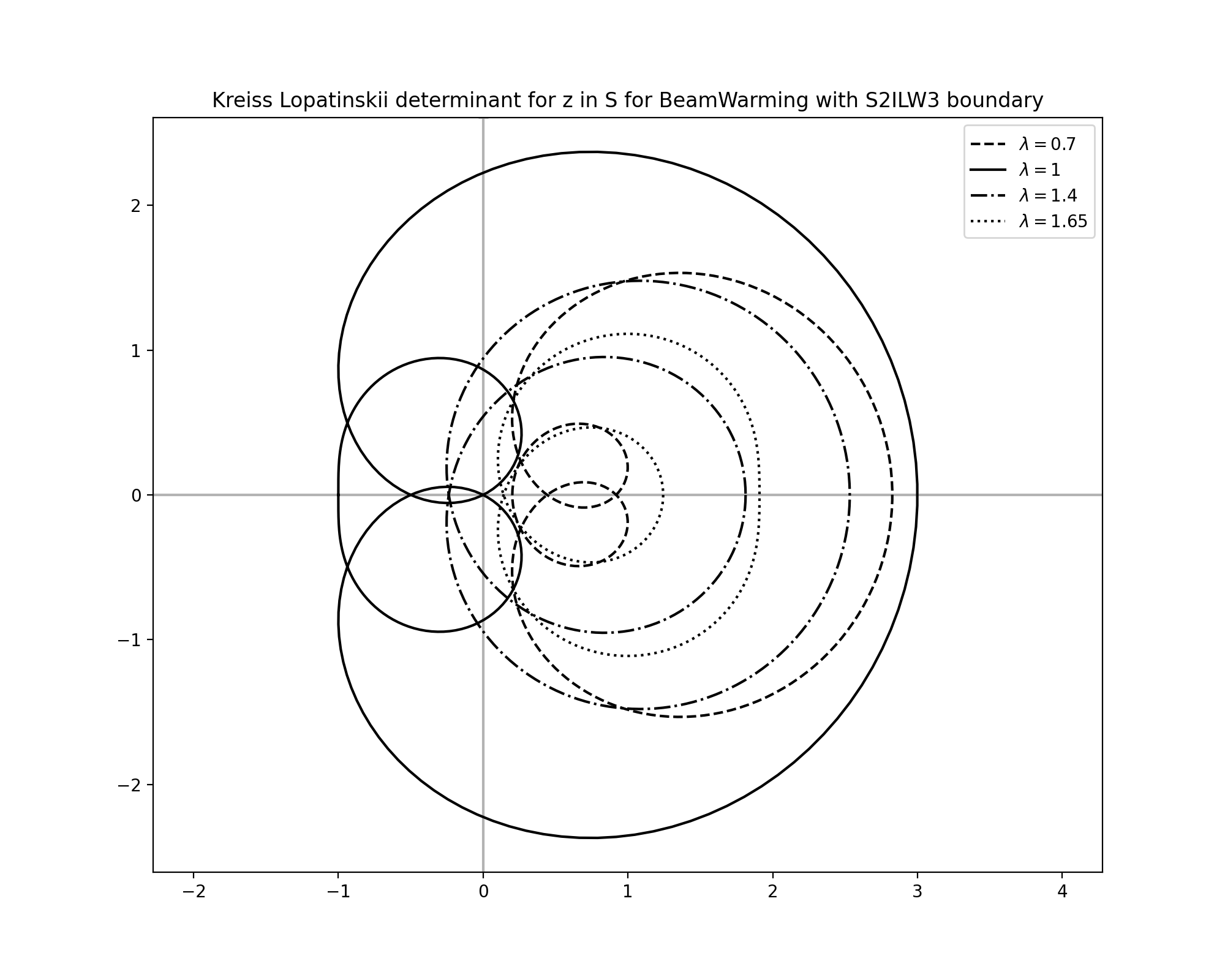}
    \caption{Rescaled Kreiss-Lopatinskii determinant $\frac{a_{-2}^2 \DKLindep}{z^2}$ for $z$ in $\S$.}\label{fig:dividedDetKLBWS2ILW3}
\end{figure}

A particular situation occurs for $\lambda = 1$, since $a_{-2}=0$ and assumption~\ref{assumption:nondegenerate} fails if we consider $r=2$. In that case, the equation \eqref{eq:Ztransform} reads $\widetilde{U}_{j-1}(z) = z \widetilde{U}_j(z)$ which is the Beam-Warming scheme for $\lambda = 1$ after $\mathcal Z$-transform. Finally, in that case, we find $\det C(z) = (\tfrac{1}{2} + z(-1+z(\tfrac 1 2 -z)))$ that we must multiply by $\frac{1}{\beta^2} = \frac{1}{z^2}$ to find the Kreiss-Lopatinskii determinant (because $m = 3$, $r=1$ and $\beta = \frac{z-a_0}{a_{-1}} = z$).

All these computations can be done for different boundary conditions and after drawing the curves, the winding number can be computed, as explained in Section~\ref{sec:windingnbr}, 
to tackle stability and that the purpose of the following subsection.


\subsection{Numerical illustration}

Figure~\ref{fig:dividedDetKLBWS2ILW3} may help to tackle the stability of the scheme~\eqref{eq:BW} as we said in Section~\ref{sec:numericalprocedureeig}, indeed, as we said in Section~\ref{sec:windingnbr}, one can compute the winding number using a numerical procedure \cite{Zapata12} and draw the winding number with respect to~$\lambda$, as seen in Figure~\ref{fig:windingnumberBWS2ILW3} for the case S2ILW3. It simplifies the observation of the number of zeros of the Kreiss-Lopatinskii determinant. Hence, the numerical experiments indicate that the scheme \eqref{eq:BW} is strongly stable for $\lambda \in ]0,1[$ but also for $\lambda \in ]1.52,1,78[$ approximately, but is unstable outside these domains.

\begin{figure}[!ht]
    \centering
    \includegraphics[trim = 3.1cm 1.4cm 3.2cm 2.74cm, clip, width = 10cm]{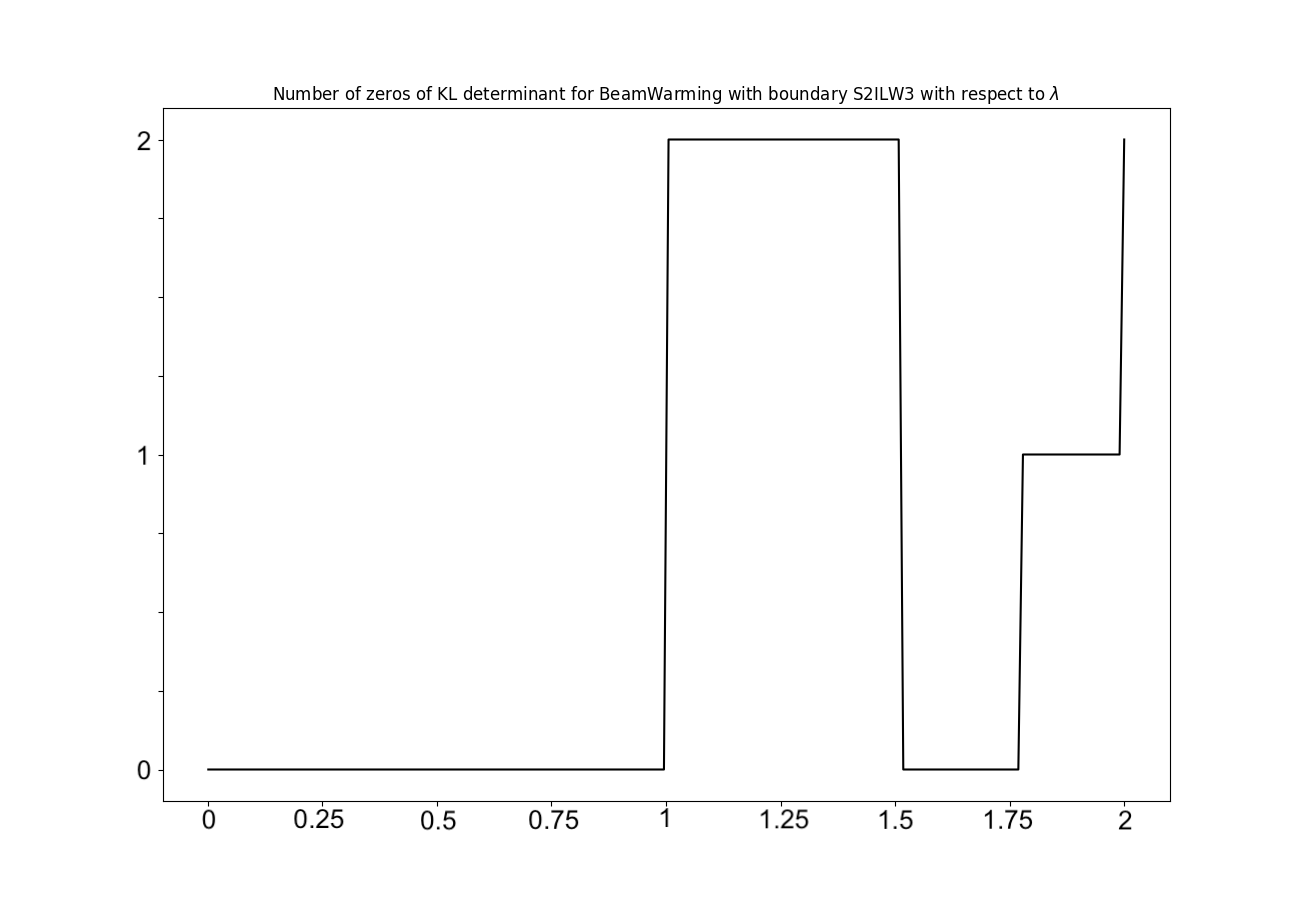}
    \caption{Number of zeros of Kreiss-Lopatinskii determinant  with respect to $\lambda$ for Beam-Warming scheme~\eqref{eq:BW} with S2ILW3 boundary condition.}\label{fig:windingnumberBWS2ILW3}
\end{figure}

Moreover, instead of taking the Y-axis to represent the number of zeros of the Kreiss-Lopatinskii determinant and having a step function, one can draw areas and compute it for other simplified inverse Lax-Wendroff boundary conditions (defined by the equation~\eqref{eq:SILWproc}) as done in Figure~\ref{fig:NbrZerosSILW}.
Note that the stability domain contains a full interval of the form $]0,\lambda_\star[$, but also another disjoint interval included in $]1,2[$ (except for the S1ILW4 scheme). This property may be used to increase the speed of the computations.

\begin{figure}[!ht]
    \centering
    \begin{tikzpicture}
        \node (A) at (0,0) {\includegraphics[trim = 6.3cm 2.1cm 4.94cm 2.95cm, clip, width = 12cm]{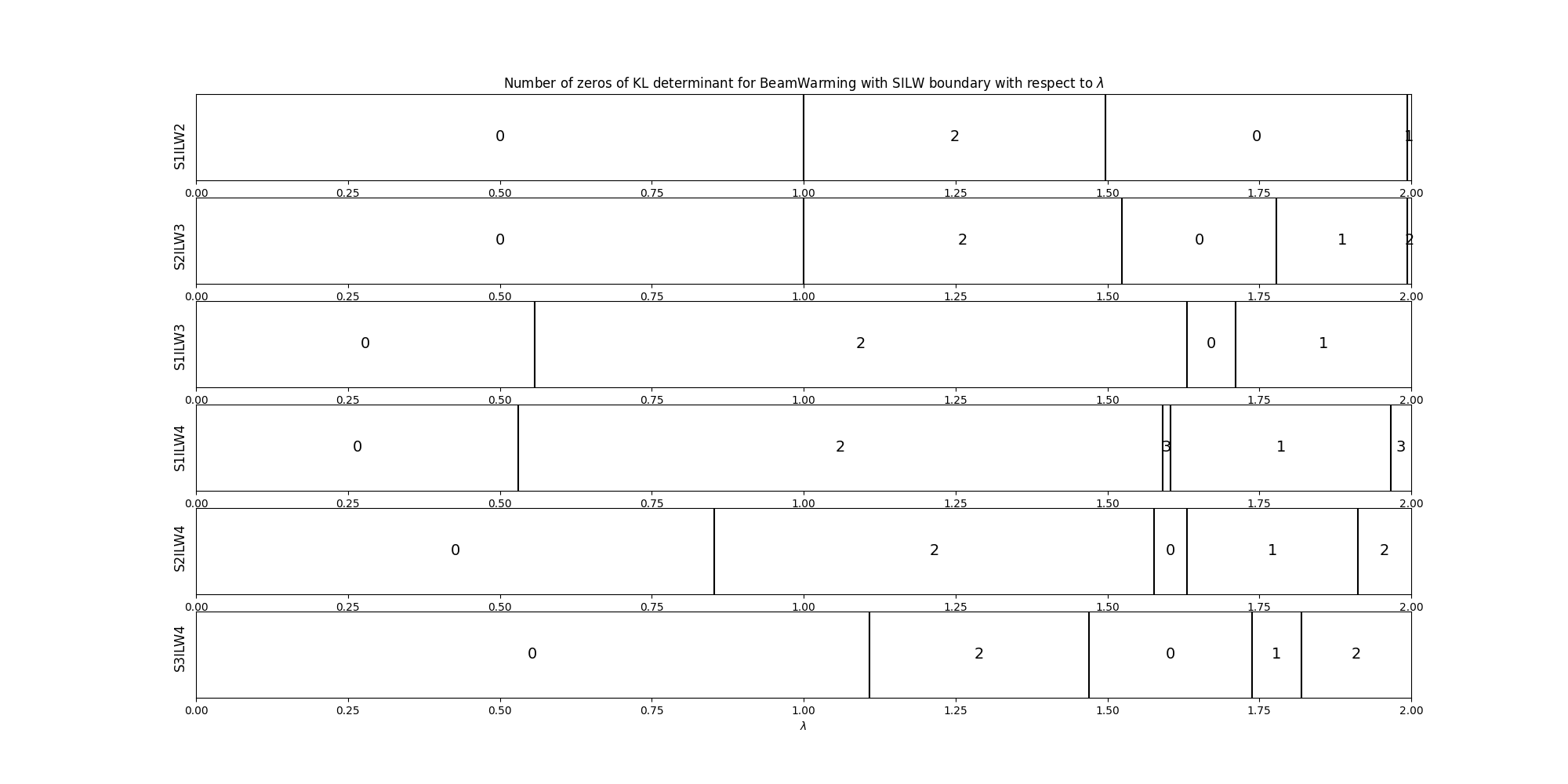}};
        \node (S1ILW2) at (-5.45,2.92) {\scriptsize{S1ILW2}};
        \node (S2ILW3) at (-5.45,1.9) {\scriptsize{S2ILW3}};
        \node (S1ILW3) at (-5.45,0.89) {\scriptsize{S1ILW3}};
        \node (S1ILW4) at (-5.45,-0.12) {\scriptsize{S1ILW4}};
        \node (S2ILW4) at (-5.45,-1.14) {\scriptsize{S2ILW4}};
        \node (S3ILW4) at (-5.45,-2.15) {\scriptsize{S3ILW4}};
        \node (l) at (0,-3.15) {\scriptsize{$\lambda$}};
    \end{tikzpicture}
    \caption{Number of zeros of Kreiss-Lopatinskii determinant for Beam-Warming scheme with different simplified inverse Lax-Wendroff boundary with respect to $\lambda$.}\label{fig:NbrZerosSILW}
\end{figure}
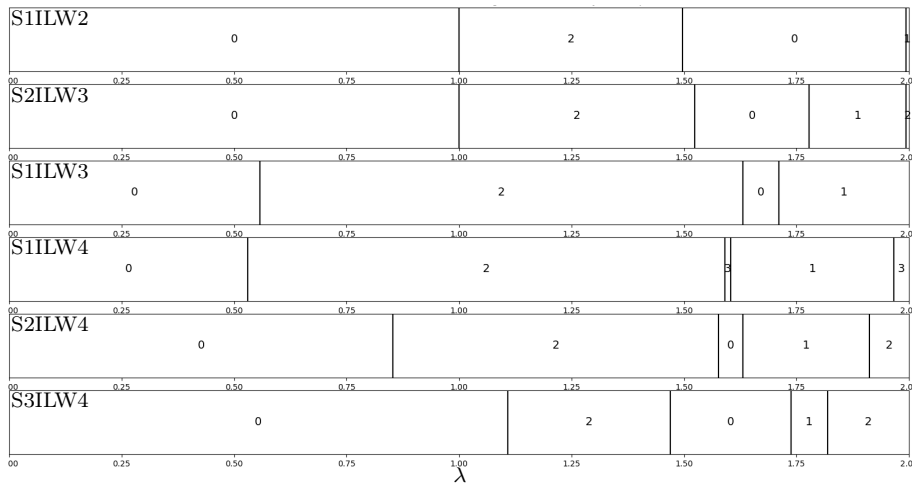

All the figures can be easily computed in Python with the common  NumPy \cite{Numpy20} library. The algorithm is really quick (less than one minute of computation achieved on a standard laptop). Moreover, our procedure provides sharp results, directly available on $\ell^2$. In particular, contrary to numerical investigations of stability which are based on the computation of the spectral radius, no arbitrary truncation of (quasi-)Toeplitz matrices is needed.


\subsection{Misalignment between boundaries and grid points}

Motivated for example by solving multidimensional problems discretized on a cartesian grids, or of one-dimensional problems  with moving boundaries as well, a usual idea consists in extrapolating the physical boundary condition to the first boundary points. This idea may be combined with the inverse Lax-Wendroff procedure in order to improve the accuracy at the boundary, see \cite{Dakin18}, \cite{Vilar15} and \cite{Li16}. As an archetype for such a situation, we consider hereafter a simple misalignment between the left physical boundary and the first numerical grid point. The advection equation \eqref{eq:advection} is set on the space domain $[x_\sigma,1]$:
\begin{align}\label{eq:advectionsigma}
    \begin{cases}
     \partial_t u + a \partial_x u = 0, &  t\>0, x \in [\borneinfintervalspatial,1],\\
    u(t,\borneinfintervalspatial) = g(t), & t\>0, \\
     u(0,x) = f(x), & x\in [\borneinfintervalspatial,1]. 
    \end{cases}
\end{align}
The space discretization $j\Delta x$ for $j\in\mathbb{Z}$, does not take into account the point $\borneinfintervalspatial$, so that one may write $\borneinfintervalspatial = (j_{\gap} + \gap) \dx$ for some integer $j_{\gap} \in \Z$ and the gap (generally nonzero) $\gap \in [-\frac 1 2, \frac 1 2[$. The scheme~\eqref{eqprincip}-\eqref{eqbord}-\eqref{eqinit} is then implemented for $j \>j_{\gap}$ only, but with $r$ ghost points at $j_{\gap} - 1,\dots, j_{\gap} - r$. For simplicity in the presentation and by translational invariance, we assume from now on that $j_{\gap} = 0$. We obtain the discretization represented on Figure~\ref{fig:reprmesh}.
\begin{figure}[!h]
\centering
    \begin{tikzpicture}
        \draw[->] (-0.2,0)--(10.2,0);
        \draw[->] (0,-0.2)--(0,1.5);
        \draw (1,-0.1)--(1,0.1);
        \draw (3,-0.1)--(3,0.1);
        \draw (6,-0.1)--(6,0.1);
        \node (A) at (1,-0.1) [below]{$\gap\dx$};
        \node (A) at (3,-0.1) [below]{$\dx$};
        \node (A) at (6,-0.1) [below]{$2\dx$};
        \node (A) at (0,-0.15) [below]{$0$};
        \draw[fill,black] (0,0) circle (0.05);
        \draw[fill,black] (6,0) circle (0.05);
        \draw[fill,black] (3,0) circle (0.05);
        \node (A) at (1,0) {$\times$};
        \node () at (10.2,0) [right]{$x$};
        \node () at (0,1.5) [above]{$t$};
        \draw[color = gray] (1,0) -- (1,1.3);
        \node[rotate = 90, color = gray] () at (1,0.65) [below]{$g(t)$};
    \end{tikzpicture}
    \caption{Representation of the mesh.}\label{fig:reprmesh}
\end{figure}
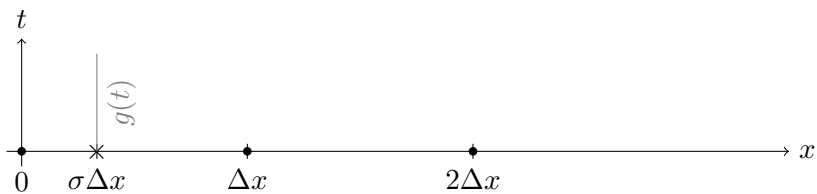

As explained above, because of the misalignment between the mesh and the boundary position, the simplified inverse Lax-Wendroff procedure~\eqref{eq:SILWproc} presented above has to be slightly adapted (see~\cite{Vilar15}). The numerical boundary condition reads
\begin{equation*}
    U_{j}^n = \sum_{k=0}^{k_d-1} \dfrac{(-(j+\gap)\dx)^k}{k!} \dfrac{g^{(k)}(n\dt)}{a^k}+ \sum_{k=k_d}^{d-1} \dfrac{(j+\gap)^k}{k!} \sum_{s=0}^{k} \tbinom{k}{s} (-1)^{k-s} U_{s}^n,\quad j \in \interval{-r}{-1}.
\end{equation*}

We perform the stability analysis of the above scheme, according to the values of both the CFL parameter $\lambda$ and the gap parameter $\sigma$. For example, with the Beam-Warming scheme~\eqref{eq:BW} supplemented with the numerical boundary condition S2ILW3 at the point $\borneinfintervalspatial$, the procedure based on the Kreiss-Lopatinskii determinant counts the number of zeros of the Kreiss-Lopatinskii determinant. The corresponding results are represented on Figure~\ref{fig:NbrZerosS2ILW3LambdaSigma}. Of course, on the line $\gap = 0$, we recover the results obtained on Figure~\ref{fig:windingnumberBWS2ILW3}.

\begin{figure}[!ht]
    \begin{minipage}[b]{0.47\linewidth}
    \begin{tikzpicture}
        \node () at (0,0) {\includegraphics[trim = 2.2cm 1.1cm 2.1cm 1.8cm, clip, width = 7.3cm]{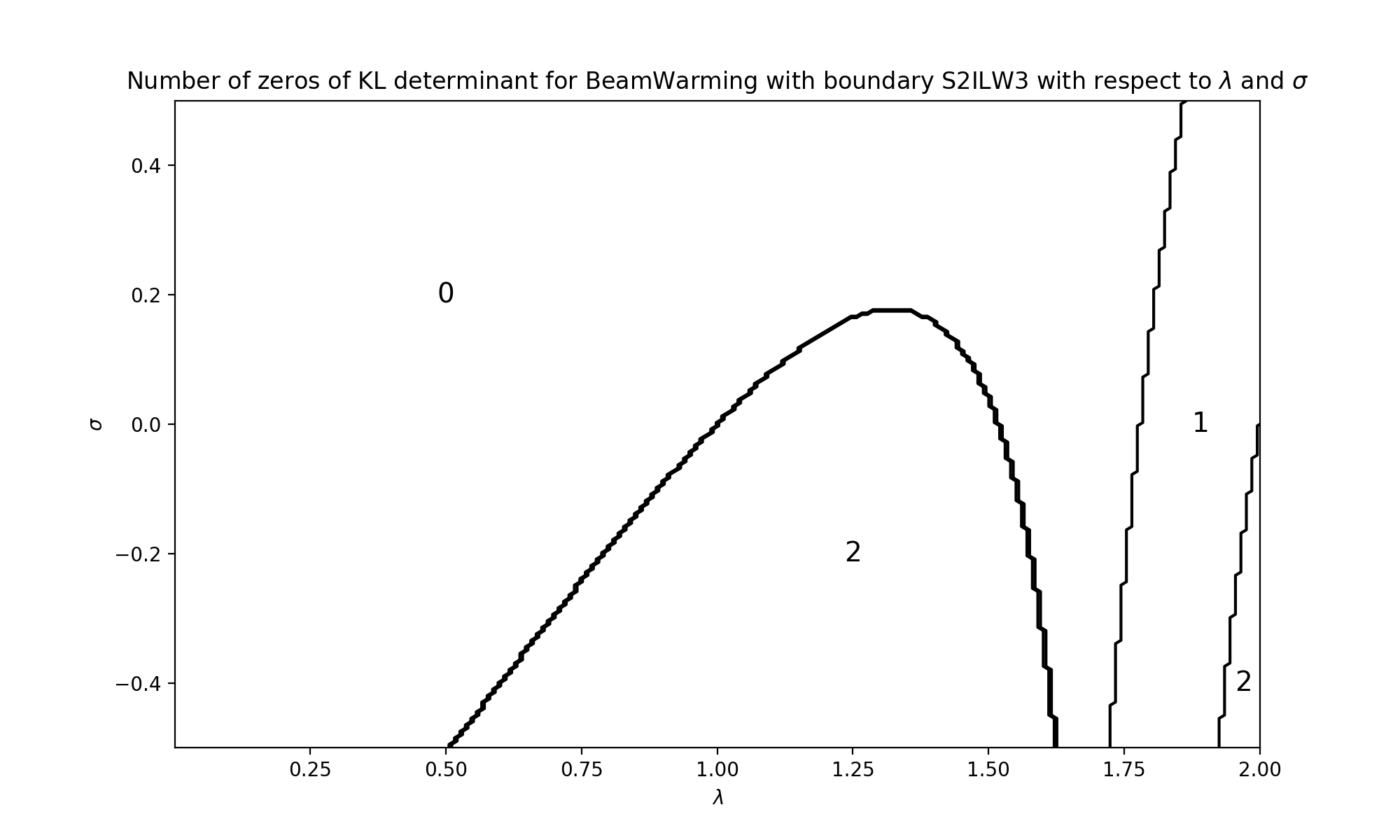}};      
            \node () at (0,0) {\includegraphics[trim = 2.2cm 1.1cm 2.1cm 1.8cm, clip, width = 7.3cm]{fig/NbrZerosS2ILW3Sigmapetit.png}};      
        \node () at (0,0) {\includegraphics[trim = 2.2cm 1.1cm 2.1cm 1.8cm, clip, width = 7.3cm]{fig/NbrZerosS2ILW3Sigmapetit.png}};      
        \node () at (3.2,-2.1) {\scriptsize{$\lambda$}};  
            \node () at (3.2,-2.1) {\scriptsize{$\lambda$}};  
        \node () at (3.2,-2.1) {\scriptsize{$\lambda$}};  
        \node () at (-3.5,1.99) {\scriptsize{$\sigma$}};  
            \node () at (-3.5,1.99) {\scriptsize{$\sigma$}};  
        \node () at (-3.5,1.99) {\scriptsize{$\sigma$}};  
    \end{tikzpicture}
\end{minipage}
\hfill
    \begin{minipage}[b]{0.48\linewidth}
    \centering
    \begin{tikzpicture}
    \node () at (0,0) {\includegraphics[trim = 2.2cm 1.1cm 2.1cm 1.8cm, clip, width = 7.3cm]{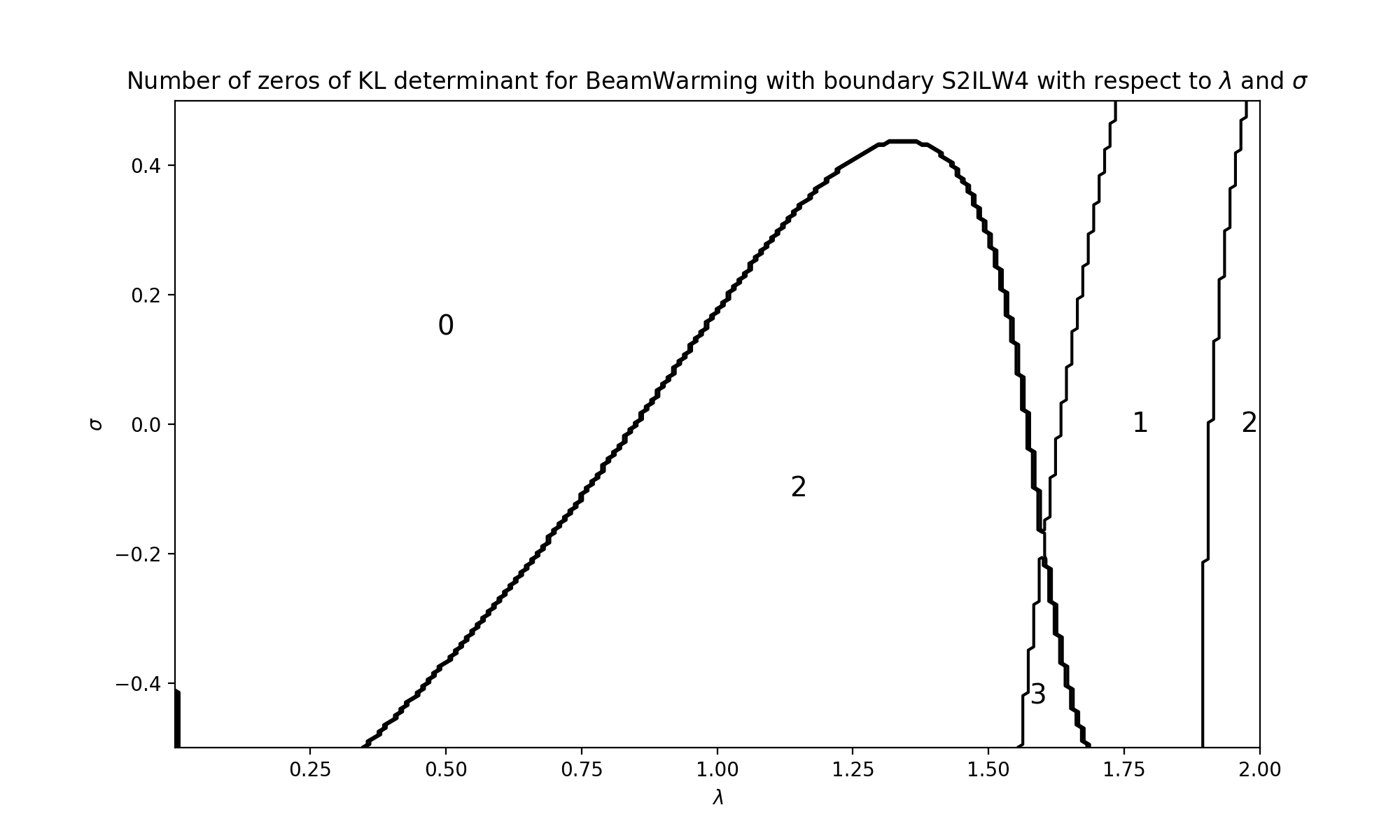}};  
        \node () at (0,0) {\includegraphics[trim = 2.2cm 1.1cm 2.1cm 1.8cm, clip, width = 7.3cm]{fig/NbrZerosS2ILW4Sigma.png}};  
    \node () at (0,0) {\includegraphics[trim = 2.2cm 1.1cm 2.1cm 1.8cm, clip, width = 7.3cm]{fig/NbrZerosS2ILW4Sigma.png}};  
    \node () at (3.2,-2.1) {\scriptsize{$\lambda$}};  
        \node () at (3.2,-2.1) {\scriptsize{$\lambda$}};  
    \node () at (3.2,-2.1) {\scriptsize{$\lambda$}};  
    \node () at (-3.5,1.99) {\scriptsize{$\sigma$}};  
        \node () at (-3.5,1.99) {\scriptsize{$\sigma$}};  
    \node () at (-3.5,1.99) {\scriptsize{$\sigma$}};  
\end{tikzpicture}
    \end{minipage}

    \caption{Stability of the Beam-Warming~\eqref{eq:BW} with S2ILW3 boundary condition (left) and with S2ILW4 boundary condition (right).}\label{fig:NbrZerosS2ILW3LambdaSigma}
\end{figure}

\bigskip

Let us now consider a very simple application of the above results, considering the advection equation in 2D on a parallelogram domain (specified later)
with a velocity field aligned with the $x$ axis. Using a cartesian grid in both directions $x$ and $y$, the numerical boundary condition will generally not coincide exactly at the grid points and the use of (S)ILW method may appear useful to maintain the order of the scheme. However, it is then mandatory to retain a CFL number for which any of the considered values for the parameter $\gap$ along the boundary belong to the stability condition.
Following the same lines of discussion as for the one-dimensional case, we consider hereafter the next problem where the direction $y$ coincide (artificially) with the parameter $\sigma$ and where the first reference grid cell is again $x_\sigma=0$ ($j_\sigma=0$).
\begin{equation*}
    \begin{cases}
        \partial_t u(t,x,y) + a \partial_xu(t,x,y) = 0, & t\>0, y \in [-1,1], x\in[y\dx,+\infty[,\\
        u(t,y,y) = g(t,y) & t\>0, y\in [-1,1],\\
        u(0,x,y) = 0 & y\in [-1,1], x\in[y\dx,+\infty[.
    \end{cases}
\end{equation*}

In the simulations, the velocity is $a=1$, the boundary condition is $g(t,y) = e^{-200(t-0.25)^2}$ and the initial datum is $f\equiv 0$. The numerical solution is computed at time $T = 0.3$  using the Beam-Warming scheme with S2ILW3, and with $N=1000$ grid points in the (truncated) $x$-direction.
The Figure~\ref{fig:numBWsigma} represents the amplitude of the numerical solution with respect to the space variable $x$ and to the gap $\sigma=y$, the discrete solution being truncated beyond the value $1$ so that unstable boundary oscillations appear as white areas.  The two black lines represents the computational domain of Figure~\ref{fig:NbrZerosS2ILW3LambdaSigma} to confront the left image of Figure~\ref{fig:NbrZerosS2ILW3LambdaSigma} and the images of Figure~\ref{fig:numBWsigma}.We observe a good agreement between the corresponding stable/unstable values of $\sigma$ in~\Cref{fig:NbrZerosS2ILW3LambdaSigma} and~\Cref{fig:numBWsigma}. 

    \begin{figure}[!ht]
        
        \begin{tabular}{m{6cm}m{6cm}}
            \begin{tikzpicture}
                \node () at (0,0) {\includegraphics[trim = 1.9cm 1.2cm 2.8cm 2.1cm, clip, width = 5.7cm]{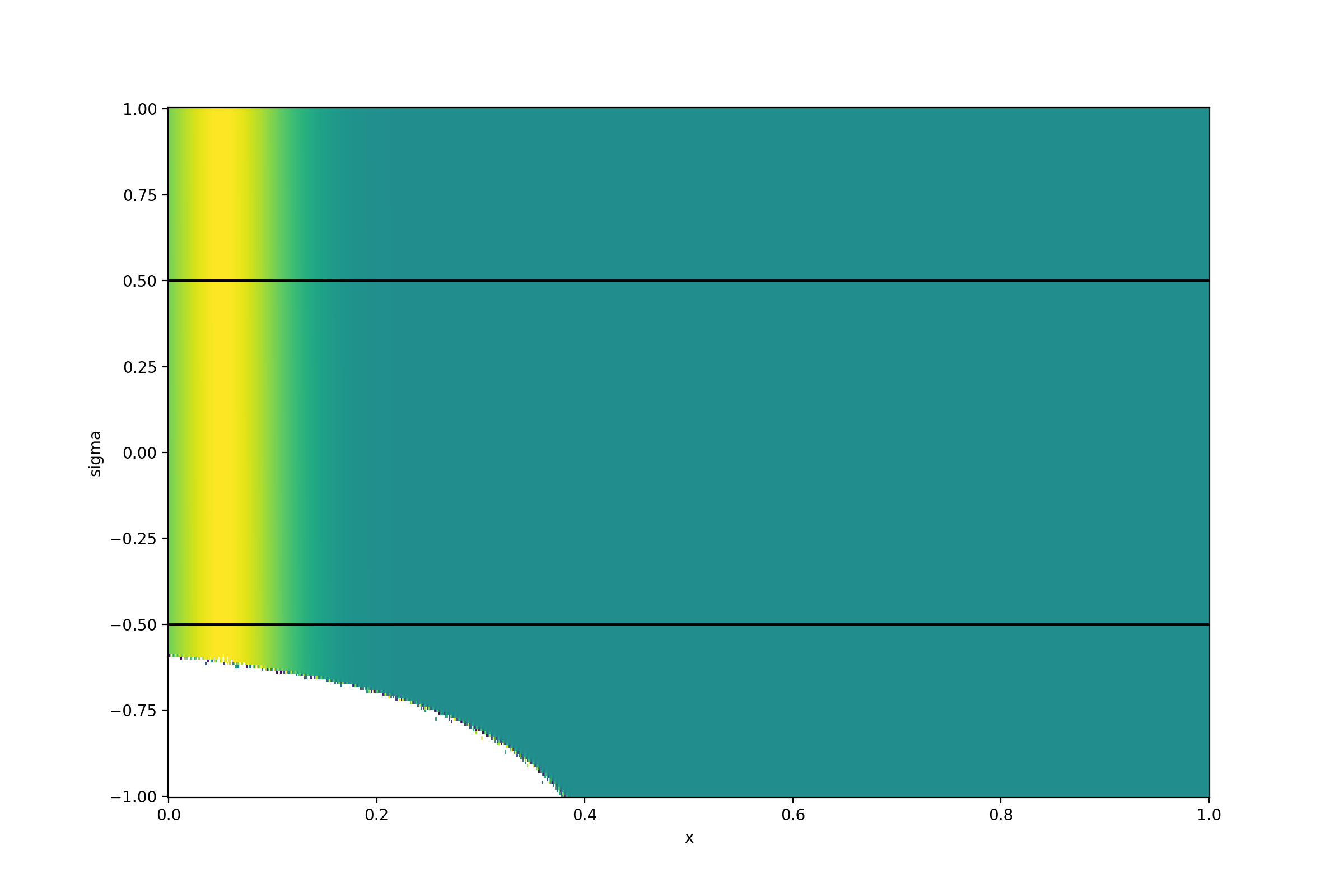}};

                \node () at (0,0) {$\lambda = 0.45$}; 
            \end{tikzpicture} & 
            \begin{tikzpicture}
                \node () at (0,0) {\includegraphics[trim = 1.9cm 1.2cm 2.8cm 2.1cm, clip, width = 5.7cm]{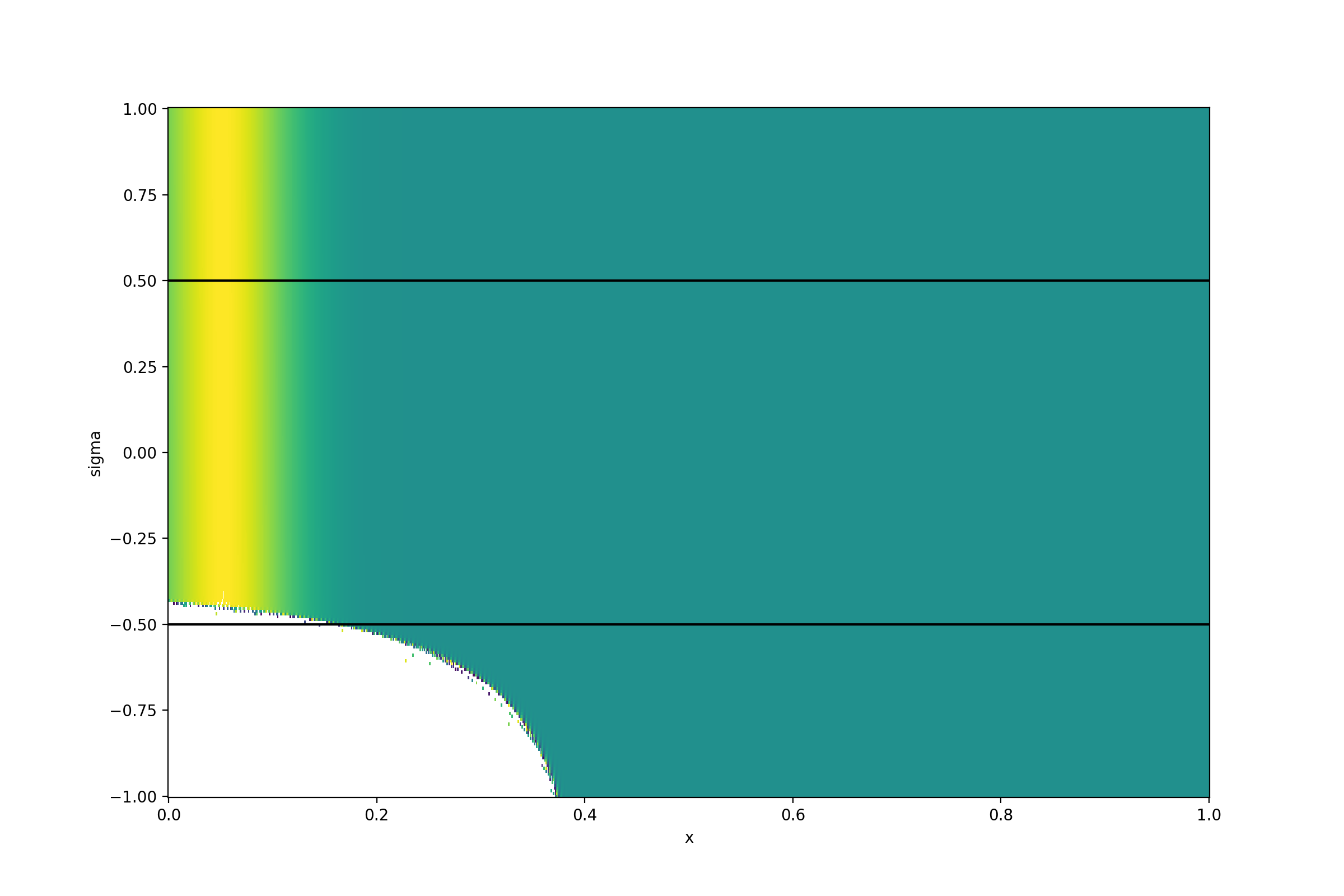}};

                \node () at (0,0) {$\lambda = 0.6$}; 
            \end{tikzpicture}
            \\
            \begin{tikzpicture}
                \node () at (0,0) {\includegraphics[trim = 1.9cm 1.2cm 2.8cm 2.1cm, clip, width = 5.7cm]{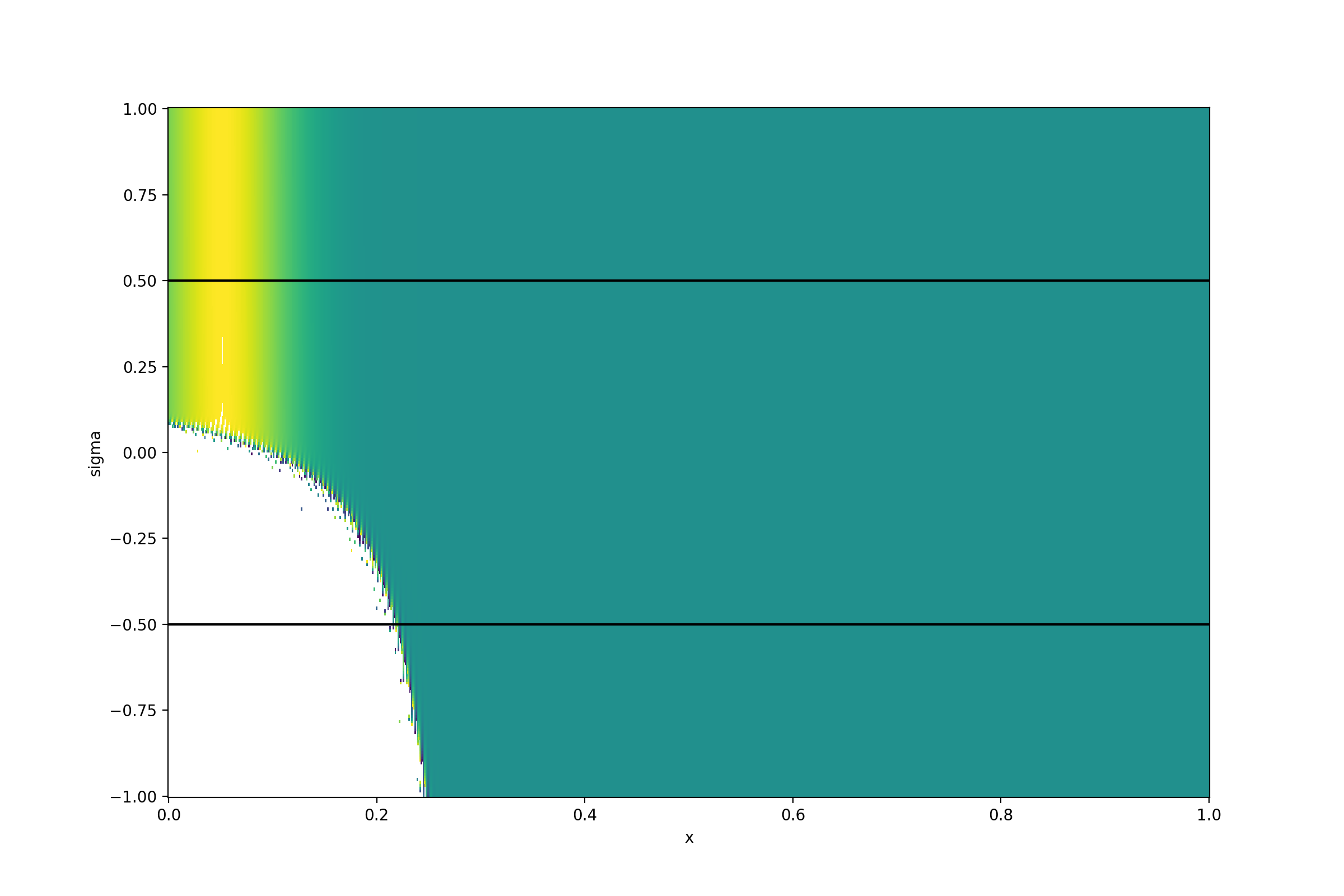}};
                \node () at (0,0) {$\lambda = 1.3$}; 
            \end{tikzpicture}
            &
            \begin{tikzpicture}
                \node () at (0,0) {\includegraphics[trim = 1.9cm 1.2cm 2.8cm 2.1cm, clip, width = 5.7cm]{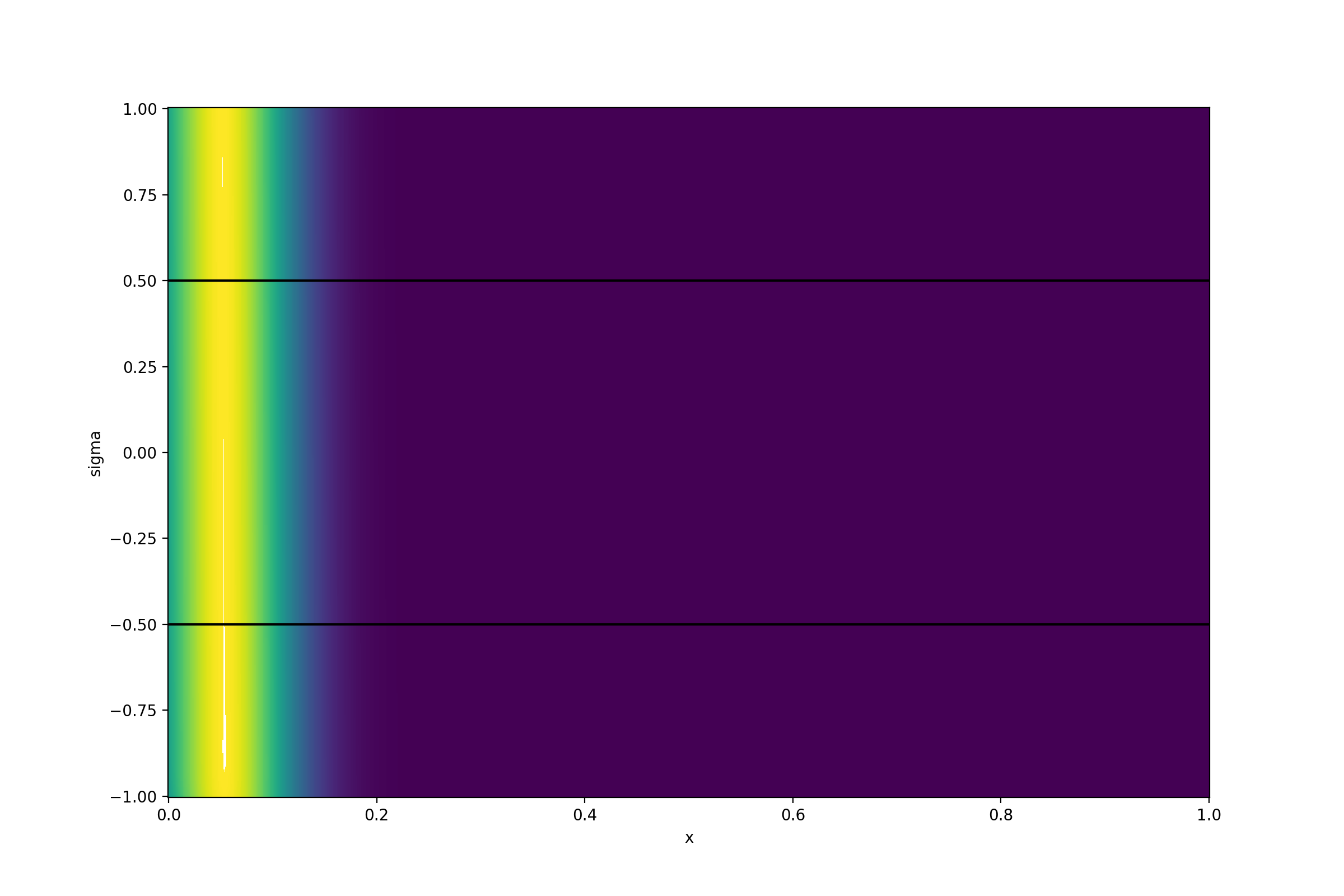}};
                \node[white] () at (0,0) {$\lambda = 1.69$}; 
            \end{tikzpicture} 
        \end{tabular}
        \caption{Numerical simulation of Beam-Warming scheme with S2ILW3 for CFL number $\lambda \in \{0.45, 0.6,1.3, 1.69\}$.}\label{fig:numBWsigma}
    \end{figure}
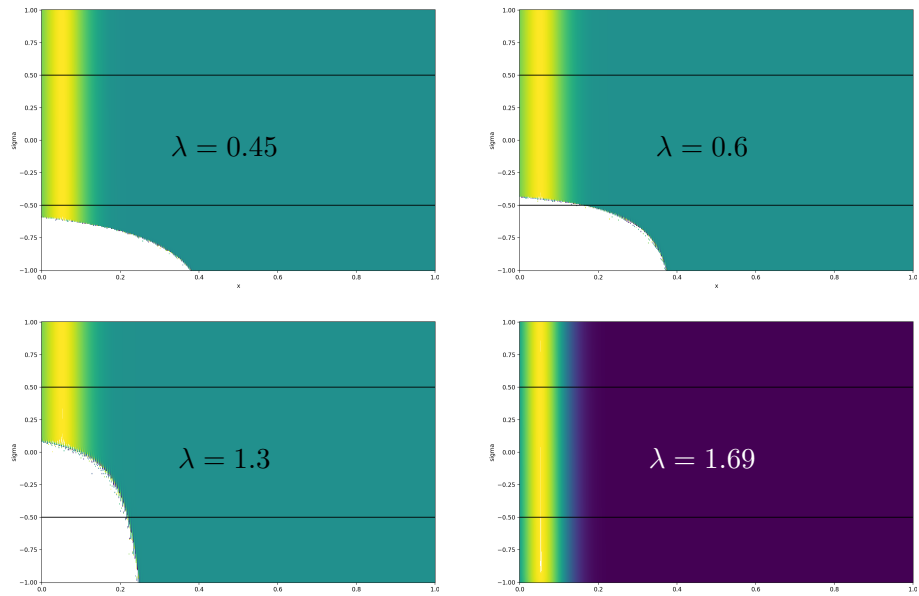


\section{Future directions}\label{sec:conclusion}

The main drawback of the present theoretical and numerical results is the restriction to the class of totally upwind schemes.
This assumption enables a specific simple analysis of the Kreiss-Lopatinskii determinant, using the explicit formula \eqref{eq:DetKLindep}, and a numerical strategy to conclude to the existence of eigenvalue or generalized eigenvalue. In this way, it answers the stability issue. 
This is only an initial effort 
on the method of designing efficient and automatic 
numerical tools for stability analysis based on the Kreiss-Lopatinskii determinant.
A first extension of the present work is the extension to the case of one-time step explicit schemes without the totally upwind 
assumption that limits the application of our approch to second-order schemes, see Iserles~\cite{Iserles83}.
Such an extension is natural but not straightforward because of the loss of Lemma~\ref{lem:algebra}: the intrinsic Kreiss-Lopatinskii determinant cannot be reduced easily into a formulation involving square matrices.
Another challenging issue is the treatment of multistep schemes and multistep boundary conditions as well. In this direction, explicit schemes may be the most practicable case because many theoretical tools remain available (Hersh, Kreiss\dots).
The difference is the dependence on $z$ in the coefficients of the characteristic equation \eqref{eq:eqcharac}. Indeed, each coefficient is a polynomial in $z$ of degree $s$ where $s$ is the number of time steps. Hence, an explicit formula of a Kreiss-Lopatinskii is more difficult to compute. In another direction, for implicit schemes or for more general boundary conditions, such as absorbing boundary conditions~\cite{Engquist77} and~\cite{Ehrhardt10} or transparent boundary conditions~\cite{Arnold03} and~\cite{Coulombel19}, it seems to be even more challenging to have a such easy-to-use theory.

\bibliographystyle{abbrv}
\bibliography{biblio}

\end{document}